\documentclass[12pt]{amsart}
\usepackage{amssymb,latexsym,amsmath,amsthm,amscd,mathrsfs,yfonts,hyperref}
\usepackage{enumerate,fullpage}
\input xy
\xyoption{all}

\newcommand{\RR}{\mathbb R}
\newcommand{\ZZ}{\mathbb Z}
\newcommand{\QQ}{\mathbb Q}

\newcommand{\NN}{\mathbb N}
\newcommand{\F}{\mathbb F}

\def\GG{\mathfrak{G}}

\def\Or{\textup{Or}}

\def\Cone{\textup{Cone}}

\def\FG{\mathtt{FG}}

\def\sub{\textup{sub}}

\def\id{\mathrm{id}}
\def\pt{\textup{pt}}

\def\Im{\textup{Im}}

\def\Wald{\mathtt{Wald}}

\def\End{\textup{End}}

\def\bK{\mathbf{K}}
\def\b1{\mathbf{1}}

\def\FSet{\mathtt{FSet}}

\def\op{\textup{op}}

\def\GS{\mathtt{FSet}^G}

\def\rk{\textup{rk}}
\def\K{\textup{K}}
\def\SS{\mathbb{S}}

\def\bC{\mathbf{C}}

\def\cP{\mathcal P}

\def\H{\textup{H}}
\def\End{\textup{End}}

\def\cC{\mathcal C}

\def\1{\bf{1}}

\def\cD{\mathcal D}

\def\G{\textup{G}}

\def\A{\mathcal{A}}
\def\bWh{\mathbf{Wh}}
\def\Wh{\textup{Wh}}

\def\cS{\mathcal S}

\newcommand{\map}{\rightarrow}
\newcommand{\functor}{\longrightarrow}

\def\Wh{\mathrm{Wh}}

\def\bG{\mathbf{G}}

\newcommand{\beq}{\begin{eqnarray}}
\newcommand{\beqn}{\begin{eqnarray*}}
\newcommand{\eeq}{\end{eqnarray}}
\newcommand{\eeqn}{\end{eqnarray*}}
\newcommand{\ul}{\underline}

\newtheorem{thm}{Theorem}[section]

\newtheorem{obs}[thm]{Observation}

\newtheorem{lem}[thm]{Lemma}
\newtheorem{prop}[thm]{Proposition}
\newtheorem{cor}[thm]{Corollary}
\newtheorem{ex}[thm]{Example}
\newtheorem{defn}[thm]{Definition}
\newtheorem{rem}[thm]{Remark}

\newtheorem{ques}[thm]{Question}

\begin{document}

\title{$\G$-theory of $\F_1$-algebras I: the equivariant Nishida problem}
\author{Snigdhayan Mahanta}

\email{snigdhayan.mahanta@mathematik.uni-regensburg.de}
\address{Fakult{\"a}t f{\"u}r Mathematik, Universit{\"a}t Regensburg, 93040 Regensburg, Germany.}
\keywords{$\K$-theory, $\G$-theory, monoids, $\lambda$-rings, $\F_1$-algebras, stable homotopy groups, equivariant sphere spectrum, assembly maps, Whitehead groups}
\subjclass[2010]{20Mxx, 19Dxx, 55Pxx}
\thanks{This research was partly supported by Australian Research Council's Discovery Projects
funding scheme (project number DP0878184) and by the Deutsche Forschungsgemeinschaft (SFB 1085).}

\begin{abstract}
We develop a version of $\G$-theory for an $\F_1$-algebra (i.e., the $\K$-theory of pointed $G$-sets for a pointed monoid $G$) and establish its first properties. We construct a Cartan assembly map to compare the Chu--Morava $\K$-theory for finite pointed groups with our $\G$-theory. We compute the $\G$-theory groups for finite pointed groups in terms of stable homotopy of some classifying spaces. We introduce certain Loday--Whitehead groups over $\F_1$ that admit functorial maps into classical Whitehead groups under some reasonable hypotheses. We initiate a conjectural formalism using combinatorial Grayson operations to address the Equivariant Nishida Problem - it asks whether $\SS^G$ admits operations that endow $\oplus_n\pi_{2n}(\SS^G)$ with a pre-$\lambda$-ring structure, where $G$ is a finite group and $\SS^G$ is the $G$-fixed point spectrum of the equivariant sphere spectrum.
\end{abstract}

\maketitle

\section{Introduction}

Let $\SS$ denote the sphere spectrum. It follows from Serre's work on the unstable homotopy groups of spheres that $\pi_n(\SS)$ is a finite abelian group for all $n\geqslant 1$. A celebrated result of Nishida says that all elements of $\oplus_{n\geqslant 1} \pi_n(\SS)$ are nilpotent \cite{Nishida}, which was conjectured earlier by Barratt. The nilpotence phenomenon became a central topic in stable homotopy theory stimulated by Ravenel's conjectures (see, e.g., \cite{Ravenel}), some of which were corroborated by Nishida's result. Most of Ravenel's conjectures were eventually solved by some ingenious methods due to Devinatz--Hopkins--Smith \cite{DevHopSmi,HopSmi}. The exterior power functor gives rise to a (pre) $\lambda$-structure on complex topological $\K$-theory. The $\lambda$-structure canonically gives rise to Adams operations, which turned out to be extremely useful for various purposes. Adams, Atiyah, Quillen, amongst others, used these operations very profitably to solve several problems in topology and representation theory. Segal demonstrated the importance of analogous operations (known as total Segal operations) in \cite{SegalOp}. Hence we believe that operations on equivariant stable homotopy are similarly interesting in their own right.

Iriye generalized Nishida's result by showing that every torsion element in $\oplus_n \pi_n(\SS^G)$ is nilpotent \cite{Iriye}, i.e., the torsion elements in $G$-equivariant ($G$ finite) stable homotopy groups of spheres are nilpotent. Note that $\pi_0(\SS^G)$ is the Burnside ring of $G$, which is known to be torsion free. Keeping in mind the utility of operations in stable homotopy theory we formulate the following Question:

\begin{ques} \label{EqNish} [Equivariant Nishida Problem]
Are there natural operations on $\SS^G$ that endow $\oplus_n\pi_{2n}(\SS^G)$ with a pre-$\lambda$-ring structure?
\end{ques}

Now we demonstrate how an affirmative answer to the above question also implies the Iriye--Nishida result mentioned above. J. Morava informed the author that G. Segal had already suggested this method to detect nilpotence and no claim to originality is made here.

Let $\oplus\pi_{2n}(\SS^G)$ denote the even graded part, which is an honest commutative ring with identity. In order to prove the nilpotence of every torsion element in $\oplus_n\pi_n(\SS^G)$, it suffices to show that every torsion element in the even graded part $\oplus \pi_{2n}(\SS^G)$ is nilpotent. Indeed, if $x\in\pi_{2m+1}(\SS^G)$ is a homogeneous torsion element of degree $(2m+1)$, then $x^2$ is a homogeneous torsion element of degree $(4m+2)$ and hence nilpotent. Since a finite sum of torsion nilpotent elements is again torsion nilpotent, the assertion follows. Any torsion element in a pre-$\lambda$-ring is known to be nilpotent (see, e.g., Lemma on page 295 of \cite{Dress}). Therefore, an affirmative answer to the Equivariant Nishida Problem recovers the Iriye--Nishida result on the nilpotence of torsion elements in $\oplus_n\pi_n(\SS^G)$.

Tits envisaged a geometry over {\em a field with one element $\F_1$} (also known as the abolute point) in \cite{Tits}, which has seen a resurgence of interest in recent years. Hence we develop our formalism in the general context of $\F_1$-geometry. Several noteworthy points-of-view on $\F_1$ and important results have appeared in the literature, for instance, those of Connes--Consani--Marcolli \cite{FunF1,CC1,CC2} (building upon some earlier work of Soul{\'e} \cite{SouleF1}), Borger \cite{Borger}, To{\"e}n--Vaqui{\'e} \cite{TVF1}, Durov \cite{Durov}, Deitmar \cite{Deitmar}, and so on. Instead of reproducing them here, we refer the readers to a survey article \cite{MapF1}. We also mention Manin's articles \cite{ManinF1,MarCycEnd}, which give a sense of the range and scope of such ideas. Our point of contact with $\F_1$-geometry is an early observation of Manin \cite{ManLECT} - one interpretation of the Barratt--Priddy--Quillen Theorem is that the $\K$-theory groups of $\F_1$ are isomorphic to the stable homotopy groups of spheres. This was, according to the author, the first indication that $\F_1$-geometry is related to stable homotopy theory. Substantial work has been done on the $\K$-theory of monoids \cite{ChuMor}, schemes \cite{CLS} and Hochschild cohomology thereof \cite{BetF1}. Within the purview of algebraic geometry the theory of monoid schemes is an active area of research now \cite{FloWei,CorHaeWalWei}. In the world of stable homotopy theory some interesting connections with $\F_1$-geometry can be found in \cite{Salch,Anevski,Scholbach}. There is a close cousin of $\K$-theory, which is called $\G$-theory (see \cite{QuiAlgK1} where it is called $\K'$-theory), and the two theories are related by a natural Cartan homomorphism $\K_n(-)\map\G_n(-)$. In this article we develop a version of $\G$-theory for $\F_1$-algebras using Waldhausen's $\K$-theory of spaces \cite{Waldhausen}. The article is organised as follows:

In Section \ref{prelim} we recall some basic facts about Waldhausen $\K$-theory and modules over $\F_1$-algebras. We follow the approach of \cite{CC2} and take an $\F_1$-algebra to simply mean a pointed monoid. In section \ref{GTheory}, using Waldhausen's machinery, we define the $\G$-theory spectrum $\bG(-)$ of an $\F_1$-algebra and set up some basic formal properties, like functoriality, transfer maps, etc.. Let $G$ be a finite group and $G_+$ denote the associated $\F_1$-algebra with a disjoint zero element. Let $\SS^G$ denote the model for the $G$-fixed point spectrum of the equivariant sphere spectrum obtained by Segal's $\Gamma$-space machine \cite{SegGamma} applied to finite pointed $G$-sets. We show that there is a weak equivalence of spectra between $\SS^G$ and $\bG(G_+)$ (see Theorem \ref{FG}). The spectrum $\bG(G_+)$ carries a natural multiplicative structure and we show that $\pi_0(\bG(G_+))\cong A(G)$, i.e., the Burnside ring of $G$. We also establish a connection between the $\G$-theory of $G_+$ and Waldhausen's $\A$-theory for $BG$ (see Proposition \ref{WhiteheadSplitting} and the Remark thereafter). In Section \ref{Assembly} we develop some further properties of $\G$-theory of $G_+$, like Mackey and Green structures. Such structures are quite useful for computational purposes - thanks to Axiomatic Induction Theory of Dress \cite{DreMackey}, they are often {\em hyperelementary computable}. For a finite group $G$ we construct a {\em Cartan assembly map} $BG_+\wedge\SS\map\bG(G_+),$ which compares the Chu--Morava $\K$-theory theory with our $\G$-theory (see Theorem \ref{main} and the Remarks thereafter). The construction of this assembly map relies on the machinery of $G$-homology theories, as developed by Davis--L{\"u}ck \cite{DavLue}. Motivated by the construction of the Whitehead spectrum of a group by Loday \cite{LodAssembly}, we define the homotopy cofibre of the Cartan assembly map to be the ($\G$-theoretic) {\em Loday--Whitehead spectrum} of $G$ over $\F_1$. There is a map of spectra $\bG(G_+)=\bG(\F_1[G])\map\bK(\mathbb{F}_q[G])$, whenever the order of $G$ is invertible in $\mathbb{F}_q$, which induces a map between the Loday--Whitehead groups $\Wh_n(G,\F_1)\map\Wh_n(G,\F_q)$. The Loday--Whitehead groups of $G$ over $\F_1$ are expressible in terms of the stable homotopy groups of the classifying spaces of some finite groups related to $G$, whereas the groups $\Wh_n(G,\F_q)$ are themselves fairly computable. In Remark \ref{RefConj} we outline a conjectural vision due to the anonymous reference that relates the Loday--Whitehead groups with the Tate cohomology groups of the sphere spectrum. Finally in Section \ref{Grayson}, using Grayson's technology \cite{GraLambda} that was enhanced by Gunnarsson--Schw{\"a}nzl \cite{GunSch}, we propose a conjectural formalism involving {\em combinatorial Grayson operations} on the $\G$-theory of $G_+$ to address Question \ref{EqNish}. We also argue that on $\G_0(G_+)\cong A(G)$ these operations recover Siebeneicher's pre-$\lambda$-ring structure on the Burnside ring of a finite group \cite{Siebeneicher}. Since $\bG(\F_1)$ is homotopy equivalent to $\SS$, these operations can be viewed as operations on stable homotopy.

\begin{rem}
Our result above concerning the pre-$\lambda$-ring structure on the Burnside ring via combinatorial Grayson operations might look promising at first sight, but actually it is deficient. While the author believes that these operations on $\G$-theory are inherently interesting, a closer inspection will reveal that they will not produce the desired pre-$\lambda$-ring structure on $\oplus_n\G_{2n}(G_+)$. Indeed, the combinatorial Grayson operations are maps $\omega^k:\G_n(G_+)\map\G_n(G_+)$. One can now readily verify that a key requirement $\omega^k(x+y) = \sum_{i=0}^k \omega^i(x)\omega^{k-i}(x)$ for a pre-$\lambda$-ring structure, will not be satisfied purely from degree considerations by setting, e.g., $x=y$ a homogeneous element in $\oplus_{n\geqslant 1}\G_{2n}(G_+)$ (unless the product structure is degenerate). However, a variant of the total Segal operation \cite{SegalOp,WalOp}, which can also be constructed on the $\G$-theory of $G_+$ thanks to \cite{GunSch}, is likely to yield better results. It is also plausible that the results in \cite{BerRog} are relevant; unfortunately the author is unable to answer the question \ref{EqNish} satisfactorily and it would be nice if someone else could step into the breach.
\end{rem}

\begin{rem}
We also propose another possible application of our conjectural formalism. Natural operations in one theory can produce operations in another theory via natural transformations. Here is an interesting scenario where this idea could be used profitably (the vertical maps are natural transformations):

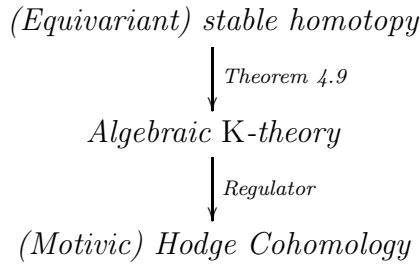
\beqn
\xymatrix{
\text{(Equivariant) stable homotopy}\ar[d]^{\text{Theorem \ref{main}}}\\
\text{Algebraic $\K$-theory}\ar[d]^{\text{Regulator}}\\
\text{(Motivic) Hodge Cohomology}
}\eeqn The combinatorial Grayson operations can be pushed down to algebraic $\K$-theory and further down via the spectrum level regulator map of Bunke--Tamme \cite{BunTam}. This mechanism might help discover some identities relating regulators.
\end{rem}

\vspace{3mm}
\noindent
{\bf Notations and conventions:} Unless otherwise stated, a (pointed) monoid or a ring is assumed to be unital but not necessarily commutative. By a module we always mean a right unital module. If $M$ is an unpointed monoid, then its associated $\F_1$-algebra with a disjoint zero or absorbing element is typically denoted by $\F_1[M]$ \cite{CC2}. For notational simplicity, we denote it by $M_+$. For a pointed object, the basepoint is referred to as $0$ or $\star$. Strictly speaking, the $\K$-theory functor should be applied only to a small Waldhausen category. We ignore such set theoretic issues, since every Waldhausen category considered here has an evident small skeleton. Although not essential for the purposes of this paper, with some foresight, occasionally we work with a specific model for spectra, called {\em symmetric spectra} \cite{HSS}. For such spectra there are two different homotopy groups - the {\em na{\"i}ve} and the {\em true} ones. It is known that for {\em semistable} symmetric spectra (see Definition 5.6.1. of ibid.) the two possible homotopy groups agree (see \cite{SchwedeHtpyGp} for an elaborate discussion). Since all symmetric spectra in sight will be semistable, we do not belabour this point here.

\vspace{3mm}
\noindent
{\bf Acknowledgements.} The author is indebted to the {\em Topological Langlands} seminar that took place at Johns Hopkins University in 2009; namely, its participants S. Agarwala, A. Banerjee, R. Banerjee, J. M. Boardman, C. Chu, C. Consani, J. Morava, A. Salch, R. Santhanam, W. S. Wilson and many others. Under various stages of development of this manuscript the author has immensely benefited from (email) communications with C. Consani, B. I. Dundas, O. Lorscheid, M. Marcolli, J. Morava, S. Schwede and C. Weibel. The author also gratefully acknowledges the hospitality of Institut des Hautes {\'E}tudes Scientifiques, where part of this project was carried out. The author is grateful to J. P. May for bringing our attention to \cite{Iriye} and to J. Rognes for indicating the relevance of \cite{BerRog}.

This article was conceptualized in 2009 at Johns Hopkins University and eventually written up in 2011 at The University of Adelaide. The project was shelved due to lukewarm response and revived only recently. A special thanks goes to the anonymous referee for writing a very enthusiastic and elaborate report full of interesting and constructive comments. Regrettably the author was unable to incorporate all the ideas generously shared by the referee into this article; nevertheless, the author hopes that those ideas would appear in the literature through contributions from others.

\section{Some preliminaries} \label{prelim}
We recall some basic facts about Waldhausen $\K$-theory and modules over $\F_1$-algebras.

\subsection{Waldhausen $\K$-theory}
Waldhausen defined a $\K$-theory functor, which is well suited for nonadditive situations. The functor takes a category with cofibrations and weak equivalences or a {\em Waldhausen category} as input and gives back a spectrum, whose homotopy groups are defined to be the $\K$-theory groups of the Waldhausen category \cite{Waldhausen}.

A {\em category with cofibrations} is a category $\cC$ with a chosen zero object $\star$ and a subcategory $c\cC$, whose morphisms are called {\em cofibrations} (denoted by $C\rightarrowtail D$), satisfying the following:

\begin{enumerate}

\item the isomorphisms in $\cC$ belong to $c\cC$,

\item for any object $C\in\cC$, the unique arrow $\star\map C$ belongs to $c\cC$, and

\item \label{axiom3} if $C\rightarrowtail D$ is a cofibration and $C\map B$ any arrow then the pushout $D\coprod_C B$ exists in $\cC$ and the canonical map $B\map D\coprod_C B$ is again a cofibration.

\end{enumerate}
A {\em Waldhausen category} $\cC$ is a category with a chosen zero object $\star$, equipped with two distinguished subcategories $c\cC$ and $w\cC$, whose arrows are called {\em cofibrations} and {\em weak equivalences} respectively. The data should satisfy some further axioms for which we refer the readers to Section 1.2 of \cite{Waldhausen}. For our purposes it suffices to say that any category with cofibrations $\cC$ as described above gives rise to a Waldhausen category by declaring the isomorphisms in $\cC$ to be the weak equivalences. The pushout of the cofibration $C\rightarrowtail  D$ along the map $C\map\star$ will be referred to as the quotient (or the cofiber) $D/C$. Diagrams of the form $C\rightarrowtail D\map D/C$ are called {\em cofibration sequences} and they play the role of `exact sequences' in this nonadditive setting. In this general setup one can subsume the traditional construction of $\K$-theory of rings and schemes.

\begin{ex} \noindent

\begin{enumerate}[(i)]

\item \label{GTheoryRing} (additive) The category of finitely generated modules over a Noetherian unital ring, where the cofibrations are monomorphisms and weak equivalences are isomorphisms.

\item \label{AlgK} (additive) The category of perfect complexes over a unital ring, where the cofibrations are monomorphisms that are split in each dergee, and the weak equivalences are quasi-isomorphisms of complexes.

\item \label{ATheory} (nonadditive) The category of finite pointed simplicial sets, where the cofibrations are injective simplicial maps and the weak equivalences are simplicial weak equivalences.
\end{enumerate}
\end{ex}
A functor $F:\cC\functor\cD$ between Waldhausen categories is called {\em Waldhausen exact} or simply {\em exact} if it preserves all the structures, i.e., zero objects, cofibrations and weak equivalences, and whenever $C\rightarrowtail D$ is a cofibration and $C\map B$ any arrow in $\cC$, the canonical map below is an isomorphism: $$F(D)\coprod_{F(C)} F(B)\map F(D\coprod_C B)$$.

Let $\Wald$ denote the category of (small) Waldhausen categories with Waldhausen exact functors as morphisms. There is an $\cS_\bullet$-construction, which produces a simplicial object in the category of Waldhausen categories with exact functors, i.e., $\cS_\bullet:\Wald\map \Wald^{\Delta^{\op}}$. The $n$-simplices $\cS_n\cC$ are $n$-step cofibration diagrams $$ C_\bullet:=\star= C_{0,0}\rightarrowtail C_{0,1}\rightarrowtail \cdots \rightarrowtail C_{0,n}$$ with explicit choices of the quotients $C_{i,j}= C_{0,j}/C_{0,i}$. The morphisms $C_\bullet\map D_\bullet$ are morphisms $C_{i,j}\map D_{i,j}$ for all $i\leqslant j$ which combine to form a morphism of diagrams in $\cC$. In particular, $\cS_0\cC$ is the trivial category with only the zero object $\star$ and the zero morphism, and $\cS_1\cC$ is isomorphic to $\cC$.

The weak equivalences in the category $\cS_n\cC$ are those, where each $C_{i,j}\map D_{i,j}$ is a weak equivalence in $\cC$. The category $w\cS_\bullet\cC$ is a simplicial category, whose category of $n$-simplices is $w\cS_n\cC$. By taking the nerve $N$ of the simplicial category $w\cS_\bullet\cC$ levelwise one obtains a bisimplicial set $N w\cS_\bullet\cC$. Now the Waldhausen $\K_i$-group of $\cC$ is defined to be the homotopy group $\pi_i$ of the based loop space of the geometric realization of this bisimplicial set, i.e., $\pi_i(\Omega|N w\cS_\bullet\cC|)$. For brevity, we write $|N w\cS_\bullet\cC|$ simply as $|w\cS_\bullet\cC|$. Waldhausen produced infinite deloopings $\Omega|w\cS_\bullet\cC|\overset{\sim}{\map}\Omega^2|w\cS_\bullet\cS_\bullet\cC|\overset{\sim}{\map}\Omega^3|w\cS_\bullet\cS_\bullet\cS_\bullet\cC|\overset{\sim}{\map}\cdots$ to exhibit a (connective) spectrum structure. In fact, Waldhausen $\K$-theory defines a functor from $\Wald$ to the category of spectra, such that any natural isomorphism between exact functors induces a homotopy between maps of spectra (stated explicitly, e.g., in 1.5.3. and 1.5.4. of \cite{ThoTro}). Later on in subsection \ref{SSstr} we shall see that the $\K$-theory spectrum actually admits a symmetric spectrum structure in the sense of \cite{HSS}. Waldhausen used this machinery to construct $\A$-theory with spectacular applications to problems in topology \cite{Waldhausen}. The following examples demonstrate the backward compatibility of this construction:

\begin{ex}\noindent
\begin{enumerate}
\item Thanks to the work of Gillet--Waldhausen \cite{Gil1}, the Waldhausen $\K$-theory of Example \eqref{GTheoryRing} produces the $\G$-theory of that ring.

\item It is known that if one feeds into this machine the Waldhausen category of Example \eqref{AlgK} above, then one recovers Quillen's algebraic $\K$-theory of that ring (see, e.g., Lemma 1.1. of \cite{WeiYao}).

\item The Waldhausen $\K$-theory of the Example \eqref{ATheory} above is called the $\A$-theory of a point.
\end{enumerate}
\end{ex}

\subsection{Modules over $\F_1$-algebras:}
There seems to be a general consensus that a right (resp. left) module over an $\F_1$-algebra (or a pointed monoid) $M$ should simply be a pointed set $S$ with a pointed monoid homomorphism $M^\op\map \End(S)$ (resp. $M\map\End(S)$). Here pointed homomorphism means that $0\in M$ much be sent to the endomorphism $S\mapsto 0_{S}$ for all $S\in S$. A pointed map $f:S_1\map S_2$ between $M$-modules is called an {\em $M$-module homomorphism} if and only if $f(sm)=f(s)m$ for all $m\in M$ and $s\in S_1$. If there is a finite subset $S'\subset S$, such that $\cup_{s\in S'} sM = S$. then it is called {\em finitely generated}. Intuitively, the elements of $S'$ form an $M$-generating set.

\begin{rem}
It is clear that if $M$ is finite then any finitely generated $M$-module is also finite. In the absence of a additive structure we are not allowed to take finite linear combinations of elements of $S$ with coefficients in $M$ .
\end{rem}

The category of unitary modules over a unital ring is additive, where finite coproduct is isomorphic to the finite product and typically denoted by the direct sum $\oplus$. In the category of $M$-modules, where $M$ is an $\F_1$-algebra, the coproduct does not agree with the product. For example, the coproduct of two $\F_1$ modules (or pointed sets) $S_1,S_2$ is given by the pointed union, i.e., $S_1\vee S_2 := S_1\times 0_{S_2} \cup 0_{S_1} \times S_2$ with the canonical induced $M$-action, which is not the same as the product. There is an operation $S_1\wedge S_2:=S_1\times S_2/\{S_1\times 0_{S_2}\coprod 0_{S_1}\times S_2\}$ that bears a similarity to tensor product of modules over a unital ring.

For any indexing set $I$ we call an $M$-module $F^I$, equipped with a set map $i: I\map F^I$, {\em free on $I$} if and only if given any set map $\iota: I\map S$, where $S$ is any $M$-module, there is a unique $M$-module map $h: F^I\map S$, such that $hi=\iota$. With this definition the module $\vee_{j=1}^n M_j$, where $M_j=M$ for all $j$, is a finitely generated and free $M$-module. Here the indexing set is $I=\{1,\cdots , n\}$ and the map $i:I\map \vee_{j=1}^n M_j$ is defined as $i\mapsto 1_{M_i}$.

In order to study $\K$-theory one needs to work with the category of finitely generated and projective modules. The notion of projectivity is rather delicate in a nonadditive situation. Presumably a lifting property can lead to a good definition, i.e., an $M$-module $S$ is {\em projective} if and only if given any surjective (at the level of pointed sets) map of $M$-modules $f:S_1\map S_2$ and any $M$-module map $g:S\map S_2$, there is a (not necessarily unique) lifting $M$-module map $h:S\map S_1$ such that $fh=g$, e.g., in \cite{Deitmar} Deitmar defines $\K$-theory of $\F_1$-algebras along these lines. A more general theory has been developed by Chu--Lorscheid--Santhanam \cite{CLS}.

There is yet another theory over Noetherian rings, called $\G$-theory, which uses the entire category of finitely generated modules. There is a natural Cartan homomorphism $\K_*(-)\map\G_*(-)$ relating the two theories, which is an isomorphism for a regular Noetherian ring. In the next section we develop a version of $\G$-theory for $\F_1$-algebras.

\section{$\G$-theory of $\F_1$-algebras} \label{GTheory}
For any $\F_1$-algebra $M$, let $\FG(M)$ denote the category of finitely generated $M$-modules with $M$-module homomorphisms. This category was studied first by Deitmar and it is an example of a {\em quasi-exact category} (see, for example, Exercise 6.16 in Chapter IV of \cite{WeibelKBook}); hence one could have applied the Q-construction to it. For our purposes we need the Waldhausen construction of $\K$-theory of $\FG(M)$. Let us define the cofibrations in $\FG(M)$ as those split monomorphisms of $M$-modules $f:S_1\map S_2$ (i.e., there is an $M$-module map $\sigma: S_2\map S_1$ with $\sigma f=\id_{S_1}$), such that $S_2/S_1$ lies in $\FG(M)$. Here $S_2/S_1$ denotes the pushout of the diagram $S_2\leftarrowtail S_1\map \ast$ in $\FG(M)$, where $\ast$ is the trivial $M$-module.

\begin{lem}
Endowed with the cofibrations just described, $\FG(M)$ becomes a category with cofibrations.
\end{lem}

\begin{proof}
The only non-trivial condition that needs to be checked is \eqref{axiom3}. The pushout $S_1\coprod_S S_2$ of the diagram $S_1\overset{f_1}{\leftarrowtail} S\overset{f_2}{\map} S_2$ is constructed explicitly as $S_1\coprod S_2/\{f_1(s)\sim f_2(s)\,|\, \forall s\in S\}$. Let $\sigma: S_1\map S$ be the splitting of $f_1$. The splitting of the canonical map $S_2\map S_1\coprod_S S_2$ is given by the map, which sends

\beqn
s &\mapsto& f_2\circ\sigma(t)\quad\text{ if $s\sim t\in S_1$,}\\
s &\mapsto& s\quad\quad\quad\quad\text{ otherwise.}
\eeqn It is readily verified that $(S_1\coprod_S S_2)/S_2$ lies in $\FG(M)$.
\end{proof}

\noindent
Now we promote $\FG(M)$ to a Waldhausen category by setting the weak equivalences to be the isomorphisms in $\FG(M)$.

\begin{defn}
The $\G$-theory spectrum of an $\F_1$-algebra $M$, denoted by $\bG(M)$, is defined to be the Waldhausen $\K$-theory spectrum of the Waldhausen category $\FG(M)$. The homotopy groups $\pi_i(\bG(M)):=\G_i(M)$ are defined to be the $\G$-theory groups of $M$.
\end{defn}

\begin{rem} \label{Gplus}
The $\G$-theory of a Noetherian unital ring is defined in terms of the exact category of finitely generated modules, where the exact sequences are not necessarily split. One can also consider a version of $\G$-theory with only split exact sequences of finitely generated modules and in the literature this is sometimes referred to as the $\G^\oplus$-theory. Strictly speaking, our $\G$-theory of $\F_1$-algebras is an analogue of this $\G^\oplus$-theory. In the sequel we are mostly going to consider situations, where this distinction is immaterial.
\end{rem}

\begin{ex}
The category $\FG(\F_1)$ is nothing but the category of finite pointed sets. It is a Waldhausen category with isomorphisms (pointed set bijections) as weak equivalences. The cofibrations are simply based injections as any injective pointed set map is split, i.e., $\iota: S_1 \rightarrowtail S_2$ is split by $\theta:S_2\map S_1$ sending $s\mapsto \iota^{-1}(s)$ if $s\in\Im(\iota)$, otherwise $s\mapsto \ast_{S_1}$. It is known that the Waldhausen $\K$-theory spectrum of $\FG(\F_1)$ is homotopy equivalent to the sphere spectrum (see, e.g., \cite{SchwedeBook}).
\end{ex}

\noindent
Let $S_1\rightarrowtail S_2\overset{p}{\map} S_2/S_1$ be a cofibration sequence in $\FG(M)$. By definition there is a splitting $M$-module homomorphism $S_2\map S_1$. We call such a cofibration sequence {\em split} if, in addition, there is an $M$-module homomorphism $\sigma: S_2/S_1\map S_2$, such that $p\circ\sigma =\id_{S_2/S_1}$, i.e., $S_2\cong S_1\vee S_2/S_1$.

\begin{lem} \label{splitcof}
Let $G$ be a group. In the Waldhausen category $\FG(G_+)$ every cofibration sequence is split.
\end{lem}

\begin{proof}
Let $S_1\rightarrowtail S_2\overset{p}{\map} S_2/S_1$ be a cofibration sequence in $\FG(G_+)$. Let us first observe that $S_2/S_1$ can be identified with $(S_2\setminus S_1)_+$. Since $G$ is a group, the pointed set $(S_2\setminus S_1)_+$ is $G_+$-invariant, i.e., it is a finitely generated $G_+$-module. The splitting of $p$ is now obtained by the basepoint preserving map $\sigma: (S_2\setminus S_1)_+\map S_2$, which sends every nonbasepoint element in $(S_2\setminus S_1)_+$ to itself in $S_2$. It is obvious that $\sigma$ is a $G_+$-module homomorphism.
\end{proof}

\noindent
Let us now address the issue of functoriality of this construction. Let us define the tensor product of a right $M$-module $S$ and a left $M$-module $S'$ to be $S\wedge_M S':= S\times S'/\{(sm, s')\sim (s,ms')\,|\, \forall m\in M, s\in S,s'\in S'\}$. Given two $\F_1$-algebras $M,N$, an $M-N$-bimodule $S$ is simultaneously a left $M$-module and a right $N$-module, such that $m(sn)=(ms)n$ for all $m\in M$, $s\in S$ and $n\in N$. If $S'$ is an $M-N$-bimodule, then there is an $N$-action on $S\wedge_M S'$ defined by $(s,s')n=(s,s'n)$ for all $n\in N$. In the sequel, we denote $S\wedge_{\F_1} S'$ simply by $S\wedge S'$. Given any $\F_1$-algebra homomorphism $\alpha:M\map N$, one can view $N$ as an $M-N$-bimodule via the homomorphism $\alpha:M\map N$ in an obvious manner. We define the base change functor $\alpha_*:\FG(M)\functor\FG(N)$ as $S\mapsto S\wedge_M N$ and $f\mapsto f\wedge\id$ (on morphisms). Note that $S\wedge_M N$ attains a canonical $N$-module structure, such that if $S'\subset S$ is a finite set of the generators of $S$ as an $M$-module, then $S'\wedge\{1_N\}$ is a finite set of generators of $S\wedge_M N$ as an $N$-module. Indeed, given any $(s,n)\in S\wedge_M N$, there are $s'\in S'$ and $m\in M$, such that $s=s'm$. Now write $(s,n)$ as $(s'm,n)=(s',\alpha(m)n)=(s',1_N)\alpha(m)n$.

\begin{prop} \label{exfunctoriality}
For any $\F_1$-algebra homomorphism $\alpha:M\map N$, the induced functor $\alpha_*:\FG(M)\functor\FG(N)$ is Waldhausen exact.
\end{prop}

\begin{proof}
The functor $\alpha_*$ clearly sends the zero object to the zero object and preserves weak equivalences, which are simply isomorphisms in the respective categories. It also preserves cofibrations as they are split inclusions. It remains to check that pushouts along cofibrations are preserved. Let $S_1\coprod_S S_2$ denote the pushout of $S_1\overset{f_1}{\leftarrowtail} S\overset{f_2}{\map} S_2$ in $\FG(M)$, where the map $S_1\overset{f_1}{\leftarrowtail} S$ is a cofibration. Then there is a canonical map

\beq \label{exact}
\alpha_*(S_1)\coprod_{\alpha_*(S)}\alpha_*(S_2)\map \alpha_*(S_1\coprod_S S_2),
\eeq induced by $\alpha_*$ applied to the maps $f_i:S_i\map S_1\coprod_S S_2$, $i=1,2$ and observing that $\Im[\alpha_*(S)\map\alpha_*(S_1\coprod_S S_2)]=\star$, where $\star$ is the zero object in $\alpha_*(S_1\coprod_S S_2)$.

It is clear that this canonical map is surjective. Now suppose $(s,n)=(s',n')\in\alpha_*(S_1\coprod_S S_2)$. Then the equality holds for a combination of the following two reasons:

\begin{enumerate}
\item $\exists$ $x\in S$, such that $f_1(x) = s$ and $f_2(x)=s'$ and $n=n'$;
\item $\exists$ factorization $n'=mn$ in $N$, such that $s=s'm$.
\end{enumerate} In either case one can easily check that the preimages must agree.
\end{proof}

\begin{cor} \label{exactFun}
The association $M\mapsto\bG(M)$ is functorial with respect to $\F_1$-algebra homomorphisms. Much like the case of rings, the association $M\mapsto\FG(M)$ is only {\em pseudofunctorial}, i.e., it respects the composition of morphisms only up to an isomorphism (see, for instance, pages 271--272 of \cite{GilletSurv}).
\end{cor}

\begin{rem}
Under normal circumstances the $\G$-theory of Noetherian schemes is only functorial with respect to flat morphisms. If $M$ is an $\F_1$-algebra, one may call an $M$-module $S$ {\em flat} if the functor $-\wedge_M S$ is {\em exact}, i.e., it commutes with all finite colimits and limits in the category of $M$-modules. One may also call an $\F_1$-algebra homomorphism $\alpha:M\map N$ {\em flat} if $N$ viewed as an $M$-module via $\alpha$ is flat. The apparent `excess functoriality' of $\G$-theory of $\F_1$-algebras is explained by Remark \ref{Gplus}.
\end{rem}

\subsection{Transfer maps}
Suppose $R\map S$ is a unital ring homomorphism such that $S$ is a finitely generated and projective $R$-module. Then there are wrong-way {\em transfer maps} $\K_i(S)\map\K_i(R)$ induced by the restriction of scalars. Its counterpart in the $\G$-theory of noetherian schemes is a covariant functoriality with respect to proper maps (see 3.16.1 of \cite{ThoTro}). There are similar {\em transfer maps} in the $\G$-theory of $\F_1$-algebras. Let $M\map N$ be an $\F_1$-algebra homomorphism, which makes $N$ into a finitely generated $M$-module. Then the restriction of scalars functor $\FG(N)\map\FG(M)$ is Waldhausen exact, whence it induces a map $\bG(N)\map\bG(M)$. For the readers' convenience, we record this fact as

\begin{lem}
If $M\map N$ is an $\F_1$-algebra homomorphism, such that $N$ becomes a finitely generated $M$-module, then there is a transfer map $\bG(N)\map\bG(M)$.
\end{lem}

\begin{ex}
Let $H\subset G$ be an inclusion of groups giving rise to an $\F_1$-algebra homomorphism $H_+\map G_+$. Then $G_+$ is finitely generated as an $H_+$-module if and only if the coset space $G/H$ is finite, i.e., $H$ is a subgroup of finite index.
\end{ex}

\subsection{The multiplicative structure} \label{SSstr}
A strict symmetric monoidal structure $-\otimes-$ on a Waldhausen category $\cC$ is called {\em biexact} if for all $A, B\in\cC$ the $\cC$-endofunctors $A\otimes -$ and $-\otimes B$ are exact and for any pair of cofibrations $A\map A'$, $B\map B'$ the induced map $A'\otimes B\coprod_{A\otimes B} A\otimes B' \map B\otimes B'$ is a cofibration. In \cite{HSS} the authors constructed a symmetric monoidal category of {\em symmetric spectra}, whose homotopy category models the symmetric monoidal stable homotopy category. The Waldhausen $\K$-theory of a Waldhausen category is a symmetric spectrum (valued in simplicial sets) in a canonical manner. Biexact functors induce a multiplication on the Waldhausen $\K$-theory spectrum and, in fact, render it with the structure of a connective and {\em quasifibrant} (in particular, semistable) symmetric ring spectrum (see Proposition 6.1.1. of \cite{GeiHes}). Recall that a symmetric spectrum $X$ is called {\em quasifibrant} if $X_n$ is a fibrant simplicial set and the adjoint to the structure map $\tilde{\sigma}:X_n\map\Omega X_{n+1}$ is a weak homotopy equivalence for all $n\geqslant 1$. It is known that a map between quasifibrant symmetric spectra (more generally, between semistable symmetric spectra) is a weak equivalence if and only if it is a $\pi_*$-equivalence, which is not true for maps between arbitrary symmetric spectra.

\begin{lem} \label{SymMon}
Let $G$ be a group. Then the symmetric monoidal bifunctor $(S,S')\mapsto S\wedge S'$ with the diagonal $G$-action on the Waldhausen category $\FG(G_+)$ is biexact.
\end{lem}

\begin{proof}
For any $S,S'\in\FG(G_+)$, the exactness of $S\wedge -$ (and $-\wedge S'$) is similar to the argument in the proof of Proposition \ref{exfunctoriality}. Suppose $S_1\rightarrowtail S$ and $T_1\rightarrowtail T$ are two cofibrations (necessarily split). There is a self-evident broken arrow in the commutative diagram below

\beqn
\xymatrix{
S\wedge T \\
& S_1\wedge T \coprod_{S_1\wedge T_1} S\wedge T_1 \ar@{-->}[ul]|-{} & S\wedge T_1 \ar[l] \ar@/_/[ull]\\
& S_1\wedge T \ar[u] \ar@/^/[uul] & S_1\wedge T_1\ar[u]\ar[l].}
\eeqn Let $S'=S/S_1$ and $T'=T/T_1$. It follows from Lemma \ref{splitcof} that $$S_1\wedge T\coprod_{S_1\wedge T_1} S\wedge T_1\cong (S_1\wedge T')\vee (S_1\wedge T_1) \vee (S'\wedge T_1),$$ $$S\wedge T\cong (S_1\wedge T')\vee (S_1\wedge T_1) \vee (S'\wedge T_1)\vee(S'\wedge T').$$ The broken arrow in the above diagram corresponds to the canonical inclusion of the first three wedge pieces. The splitting is obtained by sending the piece $S'\wedge T'$ to the basepoint (identity otherwise).

\end{proof}

\begin{prop} \label{Gmult}
For a finite group $G$, the $\bG$-theory spectrum $\bG(G_+)$ is canonically a symmetric ring spectrum. Furthermore, $\oplus\pi_n(\bG(G_+))=\oplus\G_n(G_+)$ is a graded commutative ring with identity and $\oplus\pi_{2n}(\bG(G_+))=\oplus\G_{2n}(G_+)$ is a commutative ring with identity.
\end{prop}

\begin{proof}
That $\bG(G_+)$ is a symmetric ring spectrum follows from the above Lemma \ref{SymMon} and Proposition 6.1.1. of \cite{GeiHes}. The spectrum $\bG(G_+)$ is semistable because it is a connective and convergent spectrum (see Proposition 5.6.4. (2) of \cite{HSS}). Now it follows from Proposition 6.25. of \cite{SchwedeBook} that $\oplus\pi_n(\bG(G_+))$ is a graded ring. Since the pairing on $\FG(G_+)$ is homotopy commutative, so is the ring spectrum $\bG(G_+)$. It follows that the multiplication on $\oplus\pi_n(\bG(G_+))$ is graded commutative and on $\oplus\pi_{2n}(\bG(G_+))$ is commutative.
\end{proof}

Let $G$ be a group and $R$ be a commutative unital ring. Then there is a symmetric monoidal structure on the category of $R[G]$-modules, that are finitely generated and projective over $R$, given by $-\otimes_R-$ with diagonal $G$-action. The symmetric monoidal structure of Lemma \ref{SymMon} above is similar to this one. Let $\FSet$ denote the category of finite sets. Let $G$ be any (unpointed) group. We denote by $\FSet^G$ the category of finite sets with a $G$-action and $G$-maps. When $G$ is a finite group, the $\K$-theory of the symmetric monoidal category $\FSet^G$ (under disjoint union) obtained by applying Segal's $\Gamma$-space machine is known to be (weakly) homotopy equivalent to the $G$-fixed point spectrum of the $G$-equivariant sphere spectrum, i.e., $\bK(\FSet^G)\simeq \SS^G$ (this result is presumably well-known and it is explicitly stated in Section 5 of \cite{CarDouDun}).

Let $\FSet^G_*$ denote the category of finite pointed $G$-sets with pointed $G$-set maps.  There is a functor $P:\FSet^G\functor\FSet^G_*$, sending any finite $G$-set $S$ to the pointed $G$-set $S_+$ (adding a disjoint basepoint). The category $\FSet^G$ (resp. $\FSet^G_*$) is a symmetric monoidal category with respect to disjoint union (resp. pointed union) of $G$-sets (resp. pointed $G$-sets). The functor $P:\FSet^G\functor\FSet^G_*)$ is symmetric monoidal, i.e., $P(S_1\coprod S_2)\cong P(S_1)\vee P(S_2)$. For any category $\mathcal{C}$ let $\mathcal{C}_{\cong}$ denote the underlying groupoid, i.e., the subcategory of isomorphisms. There is a canonical functor $P':(\FSet^G_*)_{\cong}\functor(\FSet^G)_{\cong}$ that simply sends a pointed finite $G$-set to the corresponding unpointed $G$-set after omitting the basepoint. This functor is strongly symmetric monoidal. Restricted to the category of isomorphisms (or underlying groupoids) the functors $P$ and $P'$ are symmetric monoidal equivalences, whence they have homotopy equivalent (connective) $\K$-theory spectra obtained by the machinery of \cite{Thomason}, for instance. By the May--Thomason uniqueness of infinite loop space machines \cite{MayTho}, this construction will agree with that of Segal's $\Gamma$-space machine. Therefore, we conclude $\bK(\FSet^G_*)\simeq\SS^G$. The functor $P$ induces a ring isomorphism between the Grothendieck groups of the symmetric monoidal categories $\FSet^G$ and $\FG(G_+)$. Let $A(G)$ denote the {\em Burnside ring} of $G$. For the benefit of the reader we record an easy observation that follows from well-known results (see, for instance, Example 5.2.2. of Chapter II in \cite{WeibelKBook}).

\begin{obs} \label{Burnside}
Let $G$ be a finite group. Then there is an ring isomorphism $A(G)\cong\G_0(G_+)$ induced by $P$.
\end{obs}

\begin{rem} \label{BurnMod}
Extrapolating this result $\G_i(G_+)$, for $i\geqslant 1$, may be regarded as the {\em higher Burnside ring} of $G$. Note that $A(G)$ is freely generated by $\{[G/H]\,|\, H\subset G \text{ subgroup}\}$ as an abelian group. It follows from Proposition \ref{Gmult} that $\G_i(G_+)$ is a module over the the Burnside ring $A(G)$ for all $i$.
\end{rem}

Recall from Section 1.8 of \cite{Waldhausen} that a {\em category with sum and weak equivalences} $\cC$ is a category with sum $-\vee -$, i.e., categorical coproduct, and weak equivalences, such that if $A_1\map A'_1$ and $A_2\map A'_2$ are weak equivalences then so is $A_1\vee A_2\map A'_1\vee A'_2$. Any cofibration category, i.e., weak equivalences = isomorphisms, admits categorical coproducts and by a forget of structure gives rise to a category with sum and weak equivalences (with a zero object). This is precisely the kind of input data to which Segal's $\Gamma$-space machine can be applied to produce a spectrum \cite{SegGamma}.

\begin{prop} \label{Segal}
Let $\cC$ be a cofibration category, viewed as a Waldhausen category. Suppose in addition that every cofibration is split, i.e., every cofibration sequence $A\rightarrowtail B\twoheadrightarrow B/A$ is isomorphic to $A\rightarrowtail A\vee B/A \twoheadrightarrow B/A$. By a forget of structure we view $\cC$ as a category with sum and weak equivalences. Then the Segal machine applied to $\cC$ (viewed as a category with sum and weak equivalences) and the Waldhausen machine applied to $\cC$ (viewed as a Waldhausen category) produce weakly homotopy equivalent spectra.
\end{prop}

\begin{proof}
Let $\mathtt{Cat}$ denote the category of small categories. Segal's machine is obtained by applying the nerve construction producing a simplicial category $N_\bullet\cC: \Delta^\op\map \mathtt{Cat}$, such that $N_n\cC=\{(A_1,\cdots , A_n, \text{ choices})\}$ after a few identifications (see, e.g., Section 1.8 of \cite{Waldhausen}). In Waldhausen's notation there is a map of spectra induced by $wN_\bullet\cC\map w\cS_\bullet\cC$ sending $$(A_1,\cdots , A_n, \text{ choices})$$ to $$(\star\rightarrowtail A_1\rightarrowtail A_1\vee A_2\rightarrowtail \cdots \rightarrowtail A_1\vee\cdots\vee A_n, \text{ (fewer) choices}).$$ Now by our assumption every object of $w\cS_n\cC$, i.e., an $n$-step filtration diagram $$(\star\rightarrowtail B_1\rightarrowtail B_2 \rightarrowtail \cdots \rightarrowtail B_n, \text{ choices})$$ is isomorphic to one of the form $$(\star\rightarrowtail A_1\rightarrowtail A_1\vee A_2\rightarrowtail \cdots \rightarrowtail A_1\vee\cdots\vee A_n, \text{ choices}).$$ It follows that for every $n$ that map $wN_n\cC\map w\cS_n\cC$ is an equivalence of categories, whence it induces a weak equivalence of simplicial sets (after applying the nerve). In other words, we have a map of bisimplicial sets, which is a levelwise weak equivalence. The assertion now follows from the Segal--Zisman Realization Lemma for bisimplicial sets (see, for instance, Lemma 5.1. of \cite{Wal1}).

\end{proof}

For any finite group $G$, the category $\FSet^G_*$ is isomorphic to the category $\FG(G_+)$. Indeed, any finitely generated module over a finite $\F_1$-algebra $G$ is necessarily finite and the zero element of $G_+$ necessarily acts as the zero morphism of the module. In fact, by forgetting the cofibration structure of $\FG(G_+)$ it can be regarded as a category with sum and weak equivalences as above. This is precisely the symmetric monoidal category $\FSet^G_*$. By Proposition \ref{Segal} the map $$wN_\bullet\FG(G_+)\map w\cS_\bullet\FG(G_+)$$ induces a weak equivalence of spectra, since the cofibrations in $\FG(G_+)$ are all split (see Lemma \ref{splitcof}). Note that the Segal machine applied to $\FSet^G_*$ produces $\SS^G$, whence we conclude

\begin{thm} \label{FG}
For every finite group, there is a weak equivalence of spectra $\SS^G\map \bG(G_+)$, i.e., $\bG(G_+)$ is a model of $\SS^G$.
\end{thm}

\begin{cor} \label{SegTD}
It follows from the results of Segal--tom Dieck (see Satz 2 of \cite{tomDieck}) that for a finite group $G$, $$\bG(G_+)\simeq \bigvee_K BW_G (K)_+\wedge\SS,$$ where $K$ runs through representatives of conjugacy classes of subgroups of $G$. Here $W_G(K)$ is the Weyl group $N_G(K)/K$ with $N_G(K)$ being the normalizer of $K$ in $G$.

\noindent
It also follows from the Iriye--Nishida result that every torsion element in $\oplus_n\G_n(G_+)$ is nilpotent.
\end{cor}

\subsection{On the relation with $\A^\oplus$-theory and $\A$-theory}
Let $G$ be a finite group and $BG$ be the simplicial classifying space. Let $R(\star,G)$ denote the category of finite pointed $G$-simplicial sets, i.e., those which are free in a pointed sense and finitely generated over $G$. It turns out that $R(\star,G)$ is a Waldhausen category with injective maps as cofibrations and weak homotopy  equivalences as weak equivalences. It follows from Theorem 2.1.5. of \cite{Waldhausen} that the Waldhausen $\K$-theory groups of $R(\star,G)$ are isomorphic to $\A_i(BG)$ (in fact, $\K_i(R(\star,G))=\A_i(BG)$ can be taken as a definition). Let $R^\oplus(\star,G)$ denote the category with sum and weak equivalences that is obtained from $R(\star,G)$ by neglecting some structure. Applying the Segal construction to $R^\oplus(\star,G)$ we obtain a spectrum, whose homotopy groups are suggestively denoted $$\A^\oplus_i(BG) := \pi_i (wN_\bullet(R^\oplus(\star,G))).$$

\begin{prop} \label{WhiteheadSplitting}
Let $G$ be a finite group. Then there is a commutative diagram of abelian groups

\beqn
\xymatrix{
\G_i(G_+)\ar[rrd]_{[G]}\ar[rr] && \A^\oplus_i(BG)\ar[d]\\
&&\G_i(G_+),
}\eeqn where the diagonal arrow $\G_i(G_+)\map\G_i(G_+)$ is multiplication by the element $[G/\{0\}]=[G]$ in the Burnside ring $A(G)$.
\end{prop}

\begin{proof}
Consider the functor

\beqn
\FG(G_+)&\map& R^\oplus(\star, G)\\
S_+ &\mapsto& S_+\wedge G_+,
\eeqn where $S_+\wedge G_+$ is viewed as a constant $G$-simplicial set with the $G$-action $(s,h)g =(sg,hg)$. There is another functor

\beqn
R^\oplus(\star,G) &\map& \FG(G_+)\\
X &\mapsto& \pi_0(X),
\eeqn where $\pi_0(X)$ is equipped with its induced $G_+$-module structure. The composition of the two functors produces the following commutative diagram in the category of (small) categories with sum and weak equivalences

\beqn
\xymatrix{
\FG(G_+)\ar[rrd]\ar[rr] && R^\oplus(\star,G)\ar[d]\\
&&\FG(G_+),
}
\eeqn where the diagonal arrow $\FG(G_+)\map\FG(G_+)$ is the composition of the other two, which sends $S_+\mapsto S_+\wedge G_+$. Applying the functor $\pi_i(wN_\bullet(-))$ and using Proposition \ref{Segal} we get a commutative diagram

\beqn
\xymatrix{
\G_i(G_+)\ar[rrd]_{[G]}\ar[rr] && \pi_i(wN_\bullet(R^\oplus(\star,G)))\cong \A^\oplus_i(BG)\ar[d]\\
&&\G_i(G_+),
}\eeqn where the diagonal homomorphism $\G_i(G_+)\map\G_i(G_+)$ is multiplication by the element $[G/\{0\}]=[G]$ in the Burnside ring $A(G)$ (see Remark \ref{BurnMod} above).
\end{proof}

\begin{rem}
If $G$ is the trivial group, then the underlying infinite loop space of $\A(BG)\cong\A(\pt)$ splits up to homotopy as $QS^0\vee \bWh^{\textup{diff}}(\pt)$, where $\bWh^{\textup{diff}}(\pt)$ is an object of fundamental interest in manifold topology. If one could identify $wN_\bullet(R^\oplus(\star,G))\simeq w\cS_\bullet(R(\star,G))$ then the above Proposition would serve as a generalization of the aforementioned splitting for non-trivial $G$. Moreover, the fact that $\A_i(BG)$ is finitely generated for all $i$ (see Theorem I of \cite{Betley}, also \cite{Dwyer}), is useful from the computational viewpoint. These are the intended applications of the above propositon. The author is extremely grateful to the referee for detecting a flaw in its proof in the original draft that necessitated the introduction of $\A^\oplus_i(BG)$.
\end{rem}

\section{Mackey and Green structure and the $\G$-theoretic assembly map} \label{Assembly}
Suppose that $G$ is a finite group and $F$ is a field, such that the characteristic of $F$ is coprime to the order of $G$. Then the canonical Cartan homomorphism from $\K$-theory to $\G$-theory
is an isomorphism, a property we frequently require in the sequel. This is more generally true if the input ring is unital regular and Noetherian. Our assumptions above imply that $F[G]$ is such a ring. However, it is an overkill for this purpose. We present a few general cases, where the group algebra is unital regular and Noetherian and the interested reader can try to adapt the machinery below to these cases. We also remark that the case of infinite groups is definitely very interesting, but necessarily more delicate.

\begin{ex}
Hall proved that $F[G]$ is Noetherian, if $G$ is polycyclic-by-finite and the characteristic of $F$ is zero \cite{HallNoethRing}. If $G$, in addition, is torsion free, then $F[G]$ is (left) regular, (see Lemma 1 of \cite{FarSni}).
\end{ex}

\begin{ex} \label{KvsG}
If $G$ is a finitely generated abelian group and $R$ is a commutative Noetherian ring, such that $qR=R$ for the order $q$ of every torsion element in $G$, then $\textup{gldim}\,R[G] = \textup{gldim}\,R + \rk(G)$ (Theorem 1.7 of \cite{ParkerGlDim}, also \cite{AusReg}). Therefore, if $F$ is a field of characteristic zero, then once again using devissage one can conclude that $\bG(F[G])\cong\bK(F[G])$.
\end{ex}

Let $\GS$ denote the the category of finite right $G$-sets with $G$-maps. A monoidal abelian category $(\A,\otimes)$-valued pair of functors $(M_*,M^*)$ on the category $\GS$ is called a {\em Mackey functor} if the following hold:

\begin{enumerate}
\item $M_*$ is covariant and $M^*$ is contravariant with $M_*(S)=M^*(S)$ for any finite $G$-set $S$,

\item For each pullback diagram in $\GS$

\beqn
\xymatrix{
U\ar[r]^F\ar[d]^H & S\ar[d]^h\\
T\ar[r]^f & V
}
\eeqn one has $F_*H^*=h^*f_*$, where $F_*=M_*(F)$, $H^*=M^*(H)$ and so on,

\item The functor $M^*$ sends finite coproducts to finite products, i.e., for any pair of $G$-sets $S,T$, the canonical inclusions of $S$ and $T$ into the disjoint union $S\coprod T$ induces an isomorphism $M^*(S\coprod T)\cong M^*(S)\oplus M^*(T)$. In particular, it follows that $M_*(\emptyset)=M^*(\emptyset)=0$ and $M_*$ sends finite coproducts to finite products.
\end{enumerate} For any arrow $f$ in $\GS$ the induced map $f_*$ (resp. $f^*$) is known as the induction (resp. restriction) homomorphism. A {\em Green functor} $M=(M_*,M^*)$ is a Mackey functor equipped with a pairing (natural transformation) $M\otimes M\map M$ satisfying certain {\em Frobenius reciprocity} conditions, such that for any finite $G$-set $S$ the abelian group $M(S)=M_*(S)=M^*(S)$ becomes a commutative unital monoid object in $\A$. It is further required that the restriction homomorphisms respect the unital monoid structure. A Mackey functor $M$ is a {\em module over a Green functor} $R$ if there is natural transformation $R\otimes M\map M$ satisfying certain {\em Frobenius reciprocity} conditions, such that for any finite $G$-set $S$ the induced map $R(S)\otimes M(S)\map M(S)$ makes $M(S)$ into a unital $R(S)$-module. For further details on Mackey and Green functors we refer the readers to \cite{DreMackey,Kuku}.

Let $S$ be any finite $G$-set. Denote by $\hat{S}$ the category whose objects are the elements of $S$ and the morphisms are triples $(s',g,s)$, such that $sg=s'$. The composition is defined as $(s'',h,s')\circ (s',g,s)=(s'',gh,s)$.

\begin{lem}
Let $G$ be a finite group. Then equipped with objectwise cofibrations and weak equivalences, the functor category $[\hat{S},\FG(\F_1)]$ (resp. $[\hat{S},\FG(F)]$) is again a Waldhausen category for any finite $G$-set $S$.
\end{lem}

\begin{proof}
This is an immediate consequence Theorem 4.2 of \cite{Kuku}.
\end{proof} Consequently, we may construct the Waldhausen $\K$-theory of $[\hat{S},\FG(\F_1)]$ and $[\hat{S},\FG(F)]$. The association $S\mapsto [\hat{S},\FG(\F_1)]$ (or $S\mapsto [\hat{S},\FG(F)]$) define a functor from $\GS\functor\Wald$.

\begin{prop}
Let $G$ be a finite group and $F$ be a field as above. Then the association $S\mapsto \K_0([\hat{S},\FG(\F_1)])$ (resp. $S\mapsto \K_0([\hat{S},\FG(F)])$) constitutes a Green functor on $\GS$. Furthermore, the association $S\mapsto \K_n([\hat{S},\FG(\F_1)])$ (resp $S\mapsto \K_n([\hat{S},\FG(F)])$) constitutes a Mackey module functor over the aforementioned Green functor.
\end{prop}

\begin{proof}
The assertions concerning the association $S\mapsto \K_n([\hat{S},\FG(F)])$ are proved in Theorem 1.4 and Theorem 1.6 of \cite{DreKuk}. The assertions concerning the association $S\mapsto \K_n([\hat{S},\FG(\F_1)])$ are proved in Theorem 5.1.3. of \cite{Kuku}.
\end{proof}

There is a canonical Waldhausen exact functor $\eta: \FG(\F_1)\map\FG(F)$, which sends a finitely generated $\F_1$-module or a finite pointed set $S_+$ to the $F$-module $F[S]$. A {\em morphism of Mackey (or Green) functors} $M=(M_*,M^*)\map N=(N_*,N^*)$ is a natural transformation simultaneously between $M_*\map N_*$ and $M^*\map N^*$ (respecting all the extra structures). The following proposition can be easily verified:

\begin{prop} \label{Mackey}
For any finite group $G$ and $F$ as above, the functor $\eta:\FG(\F_1)\map\FG(F)$ induces a morphism of Green functors $\K_0([-,\FG(\F_1)])\map\K_0([-,\FG(F)])$ and that of Mackey functors $\K_n([-,\FG(\F_1)])\map\K_n([-,\FG(F)])$ on $\GS$ for all $n\geqslant 1$.
\end{prop}

\subsection{Cartan assembly map}
We are now going to construct $\G$-theoretic assembly maps. Let us point out that, unless otherwise stated, in this section $G$ is a finite group and $F$ is a field whose characteristic is coprime to the order of $G$. We first observe the following:

\begin{lem}
Let $G$ be a finite group and $H\subset G$ be any subgroup. For $S=G/H$ there are exact equivalences between Waldhausen categories $[\hat{S},\FG(\F_1)]\cong\FG(H_+)$ and $[\hat{S},\FG(F)]\cong\FG(F[H])$.
\end{lem}

\begin{proof}
The exact equivalence $[\hat{S},\FG(F)]\cong\FG(F[H])$ is proven in Theorem 3.2 of \cite{DreKuk}. The proof for the exact equivalence $[\hat{S},\FG(\F_1)]\cong\FG(H_+)$ is similar and left to the reader.
\end{proof}

\begin{rem} \label{GtoK}
Observe that for the $G$-set $S=G/H$, where $H\subset G$ is any subgroup, we have $$\text{$\K_n([\hat{S},\FG(\F_1)])=\G_n(H_+)$ and $\K_n([\hat{S},\FG(F)])\cong\G_n(F[H])\cong\K_n(F[H])$.}$$
\end{rem}

The {\em orbit category of $G$}, denoted by $\Or(G)$, is defined to be the full subcategory of $\GS$ consisting of objects of the form $G/H$, where $H\subset G$ is a subgroup. The association $G/H=S\mapsto \bK([\hat{S},\FG(\F_1)])$ (resp. $G/H=S\mapsto \bK([\hat{S},\FG(\QQ)])$) produces a covariant module spectrum $\bK_{G,\F_1}$ (resp. $\bK_{G,F})$ over the orbit category of $G$ via the induction maps, i.e., a covariant functor from $\Or(G)$ to the category of (symmetric) $\Omega$-spectra. Thanks to the previous Lemma this covariant module spectrum over $\Or(G)$ has the property $\pi_n(\bK_{G,\F_1})(G/H)\cong \G_n(H_+)$ (resp. $\pi_n(\bK_{G,F})(G/H)\cong\G_n(F[H])$). By Lemma 4.4 of \cite{DavLue} we may now construct $G$-homology theories on the category of $G$-CW pairs by setting:

\beqn
\H^G_n(X,A;\bK_{G,\F_1})=\pi_n (\textup{map}_G(-,(X_+{\coprod_{A_+}}\Cone(A_+)))\wedge_{\Or(G)}\bK_{G,\F_1}(-)),\\
\H^G_n(X,A;\bK_{G,F})=\pi_n (\textup{map}_G(-,(X_+{\coprod_{A_+}}\Cone(A_+)))\wedge_{\Or(G)}\bK_{G,F}(-)).
\eeqn Here $-\wedge_{\Or(G)}-$ denotes the {\em balanced smash product} between a pointed (contravariant) $\Or(G)$-space and a (covariant) $\Or(G)$-module spectrum, which produces a spectrum, in the usual sense. By definition, if $X$ is a pointed contravariant $\Or(G)$-space and $Y$ is a covariant $\Or(G)$-spectrum, the the balanced product is defined to be $$X\wedge_{\Or(G)} Y = \bigvee_{G/H\in\Or(G)} (X(G/H)\wedge Y(G/H))/\sim,$$ where $\sim$ is the equivalence relation generated by $(x\phi,y)\sim (x,\phi y)$ for all morphisms $\phi:G/H \map G/K$ in $\Or(G)$ and points $x\in X(G/K)$, $y\in Y(G/H)$.

\begin{rem}
These $G$-homology theories may not possess the desirable induction structure; nevertheless, the assembly maps make sense in any $G$-homology theory.
\end{rem}

\noindent
Now the $G$-projection $EG_+\map (G/G)_+=\pt_+$ produces assembly maps $$\H_n(BG;\SS)\cong\H^G_n(EG;\bK_{G,\F_1})\map\H^G_n(\pt;\bK_{G,\F_1})\cong\G_n(G_+)$$ and $$\H_n(BG;\bK_{F})\cong\H^G_n(EG;\bK_{G,F})\map\H^G_n(\pt;\bK_{G,F})\cong\K_n(F[G]).$$

\begin{thm} \label{main}
For any finite group $G$ and a field $F$, such that the characteristic of $F$ is coprime to the order of $G$, there is a homotopy commutative diagram of spectra

\beqn
\xymatrix{
BG_+\wedge\SS\ar[r]\ar[d] & BG_+\wedge\bK_{F}\ar[d]\\
\bG(G_+)\ar[r] & \bK(F[G]),
}
\eeqn where the vertical maps are the assembly maps.

\end{thm}

\begin{proof}
Under the assumptions $F[G]$ is a regular Noetherian ring, whence $\bG(F[G])\simeq\bK(F[G])$. Thanks to Proposition \ref{Mackey} the exact functor $\eta:\FG(\F_1)\map\FG(F)$ induces a map of $\Or(G)$-module spectra $\bK_{G,\F_1}\map\bK_{G,F}$. It follows that there is a natural transformation between the $G$-homology theories defined by these spectra, induced by the following homotopy commutative diagram of spectra:

\beqn
\xymatrix{
\textup{map}_G(-,EG_+)\wedge_{\Or(G)}\bK_{G,\F_1}(-)\ar[r]\ar[d] & \textup{map}_G(-,EG_+)\wedge_{\Or(G)}\bK_{G,F}(-)\ar[d]\\
\textup{map}_G(-,(G/G)_+)\wedge_{\Or(G)}\bK_{G,\F_1}(-)\ar[r] & \textup{map}_G(-,(G/G)_+)\wedge_{\Or(G)}\bK_{G,F}(-)\\
}
\eeqn This diagram of balanced product spectra reduces to the diagram in our assertion.
\end{proof}

\begin{rem} \label{regulator}
For a regular noetherian ring $R$ the algebraic $\K$-theory groups of the group ring $R[G]$ are very interesting from the viewpoint of geometric topology thanks to the Farrell--Jones conjecture \cite{FarJon}. It would be nice to use the above commutative diagram to analyse $\K_i(R[G])$ (and the associated Whitehead groups, see Definition \ref{LodWhiteheadGrp})). However, if $R$ and $G$ do not satisfy certain hypotheses, then in the above Theorem the target of the assembly map will be $\G_i(R[G])$ and not $\K_i(R[G])$ (see Remark \ref{GtoK}).

The above natural transformation (horizontal maps) from stable homotopy to algebraic $\K$-theory admits a simple and transparent construction. In the next section we construct certain operations on $\G$-theory of $\F_1$-algebras. Bunke--Tamme have shown that the (higher) regulator map is induced by a map of spectra $r_{\sigma,p}:\bK_F \map \Sigma^p \mathbf{H}_\RR$ \cite{BunTam}. Our hope is that the operations on $\G$-theory can be {\em pushed down} via the above natural transformation followed by the Bunke--Tamme map $r_{\sigma,p}$ to uncover potentially new relationships between regulators. 
\end{rem}

\begin{rem}
In \cite{ChuMor} the authors developed a $\K$-theory for (pointed) monoids and computed their $\K$-theory groups $\K_i(G_+)\cong\pi_i(BG_+\wedge\SS)$ (see Corollary 4.3. of ibid.), where $G$ is a finite group. Hence the assembly map $BG_+\wedge\SS\map\bG(G_+)$ is a model for the Cartan homomorphism from $\K$-theory to $\G$-theory. 
\end{rem}

The Cartan assembly map $\K_i(G_+)\map\G_i(G_+)$ will not be an isomorphism in general. In \cite{LodWhitehead}, Loday constructed an assembly map $BG_+\wedge\bK_\ZZ\map\bK_{\ZZ[G]}$ and defined the homotopy cofibre to be the {\em Whitehead spectrum} $\bWh(G,\ZZ)$ of $G$ over $\ZZ$. The homotopy groups of this spectrum are called the (higher) {\em Whitehead groups} of $G$ over $\ZZ$. Analogously, the Whitehead spectrum of $G$ over $F$, such that the order of $G$ is invertible in $F$, is defined to be $\bWh(G,F):=\textup{hocofib}[BG_+\wedge\bK_F\map\bK_{F[G]}]$. Motivated by this construction, we define

\begin{defn} \label{LodWhiteheadGrp}
For a finite group $G$, the ($\G$-theoretic) Loday--Whitehead spectrum of $G$ over $\F_1$, denoted by $\bWh(G,\F_1)$, is defined to be the homotopy cofibre of the Cartan assembly map $$\textup{hocofib}[BG_+\wedge\SS\map\bG(G_+)].$$ We refer to its homotopy groups, denoted by $\Wh_n(G,\F_1)$, as the Loday--Whitehead groups of $G$ over $\F_1$, which measure the deviation of the Cartan assembly map from being an isomorphism.
\end{defn}

\begin{rem}
We have already observed that $\bG(\F_1)\simeq\SS$. Rather suggestively, one may write the above assembly map as $BG_+\wedge\bG_{\F_1}\map\bG_{\F_1[G]}$, where $\F_1[G]= G_+$. From Corollary \ref{SegTD} the spectrum $\bWh(G,\F_1)$ can be expressed as a finite wedge of suspension spectra of classifying spaces of finite groups. 
\end{rem}

\begin{lem}
For a finite group $G$, the groups $\Wh_n(G,\F_1)$ are all finitely generated (and the higher ones are finite).
\end{lem}

\begin{proof}
The assertion follows from Theorem \ref{FG} and the known results about the finiteness of stable homotopy of $BG$, obtained by a spectral sequence argument.
\end{proof}

Let $\bC(G,F)$ denote the homotopy cofibre of the map $\bG(G_+)\map\bK(F[G])$ in the above Theorem \ref{main}. We denote its homotopy groups $\pi_i(\bC(G,F))=C_i(G,F)$.

\begin{prop} \label{longexact}
Let $G$ be a finite group and $F$ be a finite field, whose characteristic does not divide the order of $G$. Then, for all $i\geqslant 2$, there is an exact sequence $$0\map C_{2i}(G,F)\map\G_{2i-1}(G_+)\map\K_{2i-1}(F[G])\map C_{2i-1}(G,F)\map G_{2i-2}(G_+)\map 0.$$
\end{prop}

\begin{proof}
Due to the assumptions on $F$ and $G$, the group algebra $F[G]$ is a finite dimensional semisimple $F$-algebra (Maschke's Theorem). Now by Wedderburn's Theorem and the triviality of the Brauer group of finite fields, we get $F[G]\simeq \prod_{k=1}^l M_{n_k} (F)$.  Since algebraic $\K$-theory respects finite products and it is Morita invariant, one has $\K_i(F[G])\simeq\prod_{k=1}^l \K_i(F)$. From Quillen's computation, it is known that $\K_{2i}(F)=\{0\}$ for any finite field $F$ \cite{QuiAlgK1}. The assertion now follows by inserting these trivial groups in the long exact sequence of homotopy groups arising from the homotopy fibration $$\bG(G_+)\map\bK(F[G])\map \bC(G,F).$$

\end{proof}

\begin{rem}
At the tail-end of the long exact sequence one finds $$0\map C_{2}(G,F)\map\G_{1}(G_+)\map\K_{1}(F[G])\map C_{1}(G,F)\map G_{0}(G_+)\map \ZZ^l\map C_0(G,F).$$ It is also known that $\K_{2i-1}(\mathbb{F}_q)\simeq \ZZ/(q^i-1)$ for all $i\geqslant 1$. However, in order to extract information about $\G_i(G_+)$ from the above sequence one needs a bit more information, which we are unable to provide at the moment.

\end{rem}

\begin{prop}
Let $G$ be a finite group and $F$ be a finite field, whose characteristic does not divide the order of $G$. Then, for all $i\geqslant 2$, there is an exact sequence $$0\map \Wh_{2i}(G,F)\map\H_{2i-1}(BG;\bK_F)\map\K_{2i-1}(F[G])\map \Wh_{2i-1}(G,F)\map \H_{2i-2}(BG;\bK_F)\map 0.$$
\end{prop}

\begin{proof}
It follows from the vanishing of $\K_{2i}(F[G])$ for $i\geqslant 1$ as argued above (see the proof of Proposition \ref{longexact}) and from the fact that $BG_+\wedge\bK_F\map\bK_{F[G]}\map\bWh(G,F)$ is a homotopy fibration (by construction).
\end{proof}

\begin{rem}
Once again, at the tail-end of the long exact sequence one finds $$0\map \Wh_{2}(G,F)\map\H_{1}(BG;\bK_F)\map\K_{1}(F[G])\map \Wh_{1}(G,F)\map \H_{0}(G;\ZZ)\map \ZZ^l\map\Wh_0(G,F).$$

\end{rem}

Since $\H_*(BG;\SS)$ (resp. $\H_*(BG;\bK_F)$) is the value of a generalized homology theory on $BG_+$, it is computable by the first quadrant homological Atiyah--Hirzebruch spectral sequence, whose $E^2_{r,s}$-terms looks like $\H_r(BG;\pi_s(\SS))$ (resp. $\H_r(BG;\K_s(F))$). The $\K$-theoretic spectral sequence $\H_r(BG;\K_s(F))$ is particularly accessible to computation, since the coefficients $\pi_s(\bK_F)=\K_s(F)$ are completely known.

\begin{rem} \label{RefConj}
We outline below an important conjectural vision due to the anonymous referee that provides a better description of the Loday--Whitehead groups. Consider the following diagram of spectra (that possibly does not commute):

\beqn
\xymatrix{
BG_+\wedge \SS\simeq \SS_{\mathrm{h}G}\ar[r]^{\quad\quad\quad(3)}\ar[d]_{(1)} & \SS^{\mathrm{h}G}\ar[r] & \SS^{\mathrm{t}G}\\
\bG(G_+)\simeq \SS^G\ar[ru]_{\quad(2)}\ar[d] \\
\bWh(G,\F_1),
}
\eeqn where the top horizontal and left vertical sequences are homotopy cofiber sequences. Now identify map (1) with the Cartan assembly map that we constructed above using the Weiss--Williams {\em universal} property of assembly maps \cite{WeiWil}. Here (2) is the canonical map from fixed points to homotopy fixed points, whose deviation from being a weak equivalence is an example of Thomason's homotopy limit problem. The Segal conjecture in stable cohomotopy asserts that this map is an isomorphism after completion with respect to the augmentation ideal of the Burnside ring $A(G)\cong \pi_0(\SS^G)$. If one could identify (at least up to homotopy) the composite of the maps (1) and (2) with the norm map (3) from homotopy orbit to homotopy fixed points then the Loday--Whitehead groups would be isomorphic, after the relevant completion, to the Tate cohomology groups of the Tate spectrum $\SS^{\mathrm{t}G}$. Here the Tate spectrum $\SS^{\mathrm{t}G}$ is the homotopy cofiber of the norm map. On an optimistic note it must be stated that an analogous assertion for $\K$-theory of rings (with trivial $G$-action) is known to be true \cite{Malkiewich}.
\end{rem}

\section{Combinatorial Grayson operations on $\G$-theory} \label{Grayson}
The constructions in this section are motivated by the well-known $\lambda$-operations on $\K$-theory. For the benefit of the reader we briefly sketch the construction of the $\lambda$-structure on higher $\K$-theory of a commutative unital ring (following Kratzer--Quillen).

\subsection{$\lambda$-operations on higher $\K$-theory} \label{KraQui}
For every natural number $k$, $\lambda^k$ is an operation on the higher algebraic $\K$-theory, which is nonadditive on $\K_0$. Recall that a commutative unital ring $R$ is called a {\em pre-$\lambda$-ring} if it is equipped with operations $\{\lambda^k\}_{k\in\NN}$ satisfying:

\begin{itemize}
\item $\lambda^0(x) = 1$ and $\lambda^1(x) = x$,
\item $\lambda^k(x+y) = \sum_{i=0}^k \lambda^i(x)\lambda^{k-i}(x)$.
\end{itemize} If the $\lambda$-operations on a pre-$\lambda$-ring $R$ satisfy, in addition, the following conditions:

\begin{itemize}
\item $\lambda^k(1) = 0$ for all $k\geqslant 2$,
\item $\lambda^k(xy) = P_k(\lambda^1(x),\cdots,\lambda^k(x),\lambda^1(y),\cdots , \lambda^k(y))$,
\item $\lambda^k(\lambda^l(x)) = P_{k,l}(\lambda^1(x),\cdots ,\lambda^{kl}(x))$
\end{itemize} then it is called a {\em $\lambda$-ring}. Here $P_k$ and $P_{k,l}$ are universal polynomials with integral coefficients, which are intimately related to symmetric functions \cite{AtiTal}. A ring homomorphism between two $\lambda$-rings, which commutes with all the $\lambda$-operations is called a {\em $\lambda$-homomorphism}. Let $G$ be any group and $A$ be any commutative ring with identity. The $\lambda$-operations on higher $\K$-theory of $A$ are constructed by the following steps:

\begin{enumerate}

\item Let $R_G(A)$ denote the Grothendieck group of the exact category of $G$-representations on finitely generated and projective $A$-modules. It attains a commutative unital ring structure (Swan ring) via the tensor product of representations with diagonal $G$-action. In \cite{SwanSplitting} Swan showed that the maps

\beqn
\lambda^k : R_G(A) &\map & R_G(A) \\ {[V]} &\mapsto & [\wedge^k V]
\eeqn define a $\lambda$-ring structure on $R_G(A)$.

\item For any $G$-representation on a finitely generated and projective $A$-module $V$, one has a homomorphism $G\map GL(A)$, which is unique up to conjugation. This gives rise to a continuous map $r(V):BG\map BGL(A)^+$, which produces a well-defined homomorphism $r: R_G(A)\map [BG,BGL(A)^+]$.

\item Let $X$ be any connected pointed finite CW complex. Then $[X,BGL(A)^+]$ admits a commutative ring structure coming from the $H$-group structure on $BGL(A)^+$ \cite{LodAssembly}. Using the homomorphism $r$ of the previous step one constructs maps $$\lambda^k: BGL(A)^+\map BGL(A)^+,$$ which are well-defined up to homotopy (cf. Section 5 of \cite{Kratzer}).

\item Let us set
\beqn
\lambda^k: [X,BGL(A)^+] &\map & [X,BGL(A)^+] \\ {[g]} &\mapsto & [\lambda^k\circ g]
\eeqn Owing to the previous step these are well-defined set maps.

\item The Kratzer--Quillen Theorem says that, equipped with the above structures, $\K_0(A)\times [X,BGL(A)^+]$ becomes a $\lambda$-ring \cite{Kratzer,Hiller}. It is canonically a $\K_0(A)$-augmented $\lambda$-ring, i.e., the projection $\K_0(A)\times [X,BGL(A)^+]\map \K_0(A)$ is a $\lambda$-homomorphism.

\item Setting $X= S^n$ for $n\geqslant 1$, we obtain the $\lambda$-structure on higher $\K$-theory. Since for $n\geqslant 1$ $S^n$ is a co-$H$-group the product structure on $\K_n(A)=[S^n,BGL(A)^+]$ is trivial (see Lemma 5.2 of \cite{Kratzer}). It follows that for all $n\geqslant 1$ the maps $\lambda^k:\K_n(A)\map\K_n(A)$ are group homomorphisms; however, they are not group endomorphisms on $\K_0(A)$.
\end{enumerate}

\subsection{Combinatorial Grayson operations}
Let $\cC$ denote the Quillen exact category of finitely generated and projective modules over a commutative and unital ring. One could try to define the $\lambda$-operations directly on the $\cS_\bullet$-construction $\lambda^k:\cS_\bullet\cC\map\cS_\bullet\cC$. Since $\K_i(\cC):=\pi_{i+1}(|w\cS_\bullet\cC|)$, such operations would necessarily be additive; however, we observed above that the $\lambda$-operations on $\K_0$ are not additive. Grayson overcame this difficulty in \cite{GraLambda} by using an {\em undelooped} model for $\K$-theory and defining certain combinatorial operations thereon. Gunnarsson--Schw{\"a}nzl extended Grayson's construction to Waldhausen categories satisfying some further hypotheses in \cite{GunSch}. We construct operations on $\G$-theory using the machinery of ibid.. Let us remark that the results in ibid. are much more general as the authors were motivated by the construction of the total Segal operation on Waldhausen $A$-theory; we only focus on the cases that are relevant for our purposes, where many simplifications occur. Recall that a cofibration category is a Waldhausen category, whose weak equivalences are isomorphisms. A cofibration category $\cC$ is said to have the {\em extension property}, if whenever there is a commutative diagram of (horizontal) cofibration sequences in $\cC$

\beqn
\xymatrix{
A \ar@{>->}[r]\ar@{>->}[d]^p & B\ar[r]\ar[d]^i & C\ar@{>->}[d]^q \\
A' \ar@{>->}[r] & B'\ar[r] & C'
}
\eeqn with extremal vertical arrows being cofibrations as indicated, it follows that $i$ is also a cofibration. We claim

\begin{lem}
Let $G$ be a group, so that we may view $G_+$ as an $\F_1$-domain. Then as a cofibration category $\FG(G_+)$ has the extension property.
\end{lem}

\begin{proof}
Suppose that we are given a commutative diagram of cofibration sequences in $\FG(G_+)$
\beqn
\xymatrix{
M \ar@{>->}[r]\ar@{>->}[d]^p & N\ar[r]\ar[d]^i & P\ar@{>->}[d]^q \\
M' \ar@{>->}[r] & N'\ar[r] & P'
}\eeqn with extremal vertical arrows as cofibrations. Since cofibrations in $\FG(G_+)$ are split monomorphisms, using the fact that $G$ is a group,  one may write identify $N$ (resp. $N'$) with $M\vee P$ (resp. $M'\vee P'$) as $G_+$-modules. Suppose that $p,q$ are split by $s,t$. Then $s\vee t: N'\cong M'\vee P' \map M\vee P\cong N$ is a splitting of $i$, proving that it is a cofibration in $\FG(G_+)$.

\end{proof}

\begin{rem}
Let $R$ be a Noetherian unital ring, so that $\FG(R)$ is an abelian category. We may regard $\FG(R)$ as a cofibration category, whose cofibrations are simply monomorphisms. Then $\FG(R)$ has the extension property. Indeed, given any morphism between cofibration sequences (with extremal vertical cofibrations)

\beqn
\xymatrix{
M \ar@{>->}[r]\ar@{>->}[d] & N\ar[r]\ar[d]^i & P\ar@{>->}[d] \\
M' \ar@{>->}[r] & N'\ar[r] & P'
}\eeqn simply use the Snake Lemma to deduce that $i$ is a monomorphism.

\end{rem}

Consider the symmetric monoidal cofibration category $(\cC,\wedge)$, where $\cC$ is the category of finitely generated modules $\FG(G_+)$ over an $\F_1$-algebra $G_+$ ($G$ being a group). For $S,T\in\FG(G_+)$, we equip $S\wedge T$ with the diagonal $G_+$-action. We know that in this case $-\wedge -$ is biexact and the cofibration category satisfies the extension property. Then we have the following two operations:

\begin{enumerate}
\item An analogue of the tensor product of modules
\beqn \boxtimes: \cC\times\cC &\map& \cC\\ (S,T) &\mapsto& S\wedge T \eeqn with diagonal $G_+$-action.

\item An analogue of the (generalized) exterior product of modules; let $F_k(\cC)$ denote the category of $k$-filtered objects of $\cC$, i.e., one whose objects are cofibration strings $\underline{S}:=S_1\rightarrowtail S_2\rightarrowtail \cdots \rightarrowtail S_k$. Then consider \beqn \Diamond^k : F_k(\cC) &\map& \cC\\ \underline{S} &\mapsto& \Diamond^k(\underline{S})\eeqn Now we describe the diamond functor $\Diamond^k$ rather explicitly. First consider the $k$-th diamond product $\Diamond^k S := S\Diamond \cdots \Diamond S$ ($k$ times) of a single module $S$. Consider the $G_+$-submodule $Q$ of $\boxtimes^k S$ generated by tuples $(s_1,\cdots ,s_k)$, such that $s_i = s_j$ for some pair $i,j$ and $i\neq j$. We define $\Diamond^k S:= \boxtimes^k S/Q$, i.e., the pushout of the following diagram in $\cC$ that admits finite colimits:

\beqn
\xymatrix{
Q \ar[r]\ar[d] & \text{$\boxtimes^k S$} \\
\star & .
}\eeqn 

Now given $\ul{S}:=S_1\rightarrowtail S_2\rightarrowtail \cdots \rightarrowtail S_k$ consider the canonical map $S_1\boxtimes\cdots \boxtimes S_k\map \boxtimes^k S_k$.
Then we define $\Diamond^k(\ul{S})$ to be the $G_+$-submodule of $\Diamond^k S_k$ generated by the image of the composite $$S_1\boxtimes\cdots \boxtimes S_k\map \boxtimes^k S_k \map \Diamond^k S_k.$$
\end{enumerate}

\begin{rem} \label{ksubset}
For any $S_+=S\coprod \{\star\}\in\FG(G_+)$, using Lemma \ref{splitcof} one can ascertain that $\Diamond^k S_+$ is the set of ordered subsets of $S$ (without repetition) of cardinality $k$ along with the disjoint basepoint $\star$. It carries a canonical pointed $G_+$-action. In particular; if $S$ has cardinality $n$, then it is clear that $\Diamond^k S_+=\star$ for all $k>n$.
\end{rem}

In order to construct operations on $\G$-theory from the above data, one needs further preparation. It turns out that the $\cS_\bullet$-construction is not suitable and one needs a variant called the $\textup{G}_\bullet$-construction (see \cite{GilGra}). In order to avoid notational confusion due to the abundance of $G$'s appearing in various forms, we denote this construction by $\GG_\bullet$. Much like the $\cS_\bullet$-construction, for any $\cC\in\Wald$ the $\GG_\bullet$-construction is a functor $\Wald\map \Wald^{\Delta^{\op}}$ defined by the cartesian square

\beqn
\xymatrix{
\GG_\bullet\cC\ar[r] \ar[d] & P\cS_\bullet\cC\ar[d]\\
P\cS_\bullet\cC\ar[r] & \cS_\bullet\cC,
}
\eeqn where $P\cS_\bullet\cC$ is the path object so that $P\cS_n\cC=\cS_{n+1}\cC$ and $P\cS_\bullet\cC\map \cS_\bullet\cC$ is given by the boundary map $d_0:\cS_{n+1}\cC\map\cS_n\cC$. Since $P\cS_\bullet\cC$ is simplicially homotopic to a point, there is a canonical map $|w\GG_\bullet\cC|\map \Omega |w\cS_\bullet\cC|$ which is a weak equivalence if $\cC$ is {\em pseudo-additive} (see Theorem 2.6 of \cite{GSVW}).

\begin{rem} \label{padd}
We do not reproduce the exact definition of {\em pseudo-additivity}. It suffices to say that all Quillen exact categories are pseudo-additive (see Remark 2.7 of ibid.) and so are cofibration categories $\cC$, whose cofibration sequences are split in $\cC$ (see Remark 2.4 (3) of ibid.). Therefore, for any group $G$ and any Noetherian unital ring $R$ the cofibration categories $\FG(G_+)$ and $\FG(R)$ are both pseudo-additive. Consequently, we have $$\text{$\bG(G_+)\cong |w\GG_\bullet\FG(G_+)|$ and $\bG(R)\cong |w\GG_\bullet\FG(R)|$.}$$
\end{rem}

\noindent
Grayson enlisted five conditions ($(\mathtt{E1}), \cdots, (\mathtt{E5})$), which $\boxtimes$ and $\Diamond^k$ must satisfy so that one can construct the desired operations \cite{GraLambda} (see also pages 270--271 of \cite{GunSch}). The author is unable to provide a proof of the following lemma and this makes the rest of the section {\em conjectural}.

\begin{lem}
For any group $G$, the operations $\boxtimes$ and $\Diamond^k$ defined on $\FG(G_+)$ above satisfy the conditions $(\mathtt{E1}),\cdots ,(\mathtt{E5})$.
\end{lem}

\noindent
Let $\cC$ be a cofibration category. The maps $\cC\cong (P\cS_\bullet\cC)_0\map P\cS_\bullet\cC$ and the zero map $\cC\map P\cS_\bullet\cC$ compose with $d_0$ to the same image inside $\cS_\bullet\cC$. Therefore, the pullback definition of $\GG_\bullet\cC$ produces a map $\cC\map\GG_\bullet\cC$, which can be iterated to produce maps $\GG^n\cC\map\GG^{n+1}\cC$ for all $n\in\NN$. Let $g\cC$ be the simplicial set obtained by setting $g_m\cC =\textup{Obj}\GG_m\cC$, i.e., extracting the object sets from the simplicial category $\GG_\bullet\cC$. Then there is a weak equivalence $|g\cC|\overset{\sim}{\map}|w\GG_\bullet\cC|$ if the weak equivalences in $\GG_\bullet\cC$ are all isomorphisms (see Lemma 2.14 of \cite{GSVW}).

For any Quillen exact category $\cC$ equipped with two functors $\boxtimes$ and $\Diamond^k$ satisfying the conditions $(\mathtt{E1}),\cdots ,(\mathtt{E5})$, there are simplicial Grayson maps $$\omega^k:\sub_k w\GG_\bullet\cC\map w\GG_\bullet^k\cC,$$ where $\sub_k$ denotes the $k$-fold edgewise subdivision functor. The construction of the maps $\omega^k$ are quite involved and we refer the readers to the original reference \cite{GraLambda} (or Section 2 of \cite{GunSch}). Passing to the homotopy groups of the geometric realization one obtains the operations on the $\K$-theory groups of $\cC$. Let $G$ be a group. Setting $\mathcal{M}_n=\FG(G_+)$ for all $n\geqslant 0$ in Theorem 4.1 of \cite{GunSch} we obtain

\begin{prop}
There are simplicial Grayson maps $$\omega^k:\sub_k w\GG_\bullet\FG(G_+)\map w\GG_\bullet^k\FG(G_+).$$
\end{prop}

\begin{thm} \label{Gop}
There are well-defined operations on the $\G$-theory of an $\F_1$-algebra $G_+$ ($G$ being a group) $$\omega^k:\G_i(G_+)\map\G_i(G_+)$$ for all $i\geqslant 0$, which are induced by Grayson's maps.
\end{thm}

\begin{proof}
It is well known that the canonical map $\sub_k w\GG_\bullet\FG(G_+)\map w\GG_\bullet\FG(G_+)$ is a weak equivalence after geometric realization. The assertion is more obvious after identifying $|g\FG(G_+)|\overset{\sim}{\map} |w\GG_\bullet\FG(G_+)|$. Theorem 2.8 of \cite{GSVW} gives us a weak equivalence $|w\GG_\bullet\FG(G_+)|\overset{\sim}{\map}\Omega |w\cS_\bullet\FG(G_+)|=\bG(G_+)$, since $\FG(G_+)$ is a pseudo-additive cofibration category (see Remark \ref{padd} above). It follows from Proposition 1.55' of ibid. that there is a weak equivalence $|w\GG_\bullet\cC|\overset{\sim}{\map} |w\GG_\bullet^k\cC|$ whenever $\cC$ is a pseudo-additive cofibration category (it is explicitly observed on Page 264 of ibid.). Therefore, taking the geometric realization of the map $$\omega^k:\sub_k w\GG_\bullet\FG(G_+)\map w\GG_\bullet^k\FG(G_+)$$ and passing to homotopy groups we get the desired operations.
\end{proof}

\begin{rem}
Let $\cP(R,G)$ be the category of representations of a discrete group $G$ over finitely generated and projective $R$-modules, where $R$ is a commutative unital ring. Let $M\boxtimes N$ stand for $M\otimes_R N$ with diagonal $G$-action for any $M,N\in\cP(R,G)$. Let $\ul{M}:= M_1\rightarrowtail M_2 \rightarrowtail \cdots \rightarrowtail M_k$ be a string of monomorphisms of finitely generated and projective $R$-modules and let $\Diamond^k(\ul{M})$ be the image of $M_1\otimes_R \cdots \otimes_R M_k$ in $\wedge^k M_k$, i.e., the $k$-th exterior power of $M_k$ over $R$. Then $\boxtimes$ and $\Diamond^k$ satisfy the conditions $(\mathtt{E1}),\cdots , (\mathtt{E5})$ and Grayson used the associated maps $$\omega^k:\sub_k w\GG_\bullet\mathcal{P}(R)\map w\GG_\bullet^k\mathcal{P}(R)$$ to construct operations on the higher algebraic $\K$-theory of the exact category $\cP(R,G)$. He also showed that these operations agree with the $\lambda$-operations on the higher algebraic $\K$-theory of $R$ (when $G$ is trivial) as described above (see Subsection \ref{KraQui}).
\end{rem}

\noindent
The combinatorics involved in the construction of the operations $\omega^k$ are quite complicated. It is not straightforward to verify from the definition that endowed with these operations the higher algebraic $\K$-theory attains a $\lambda$-structure. The problem is circumvented by identifying the operations with the already known $\lambda$-structure on higher algebraic $\K$-theory. The author is not aware of any $\lambda$-structure on the {\em higher Burnside ring}. However, Siebeneicher defined a $\lambda$-structure on the Burnside ring of a finite group $A(G)$ in \cite{Siebeneicher} by means of $\lambda^k: A(G)\map A(G)$, which sends a $G$-set $S \mapsto \{ T\subset S \,|\, |T|=k\}$ with its canonical $G$-action.

\begin{prop}
For a finite group $G$, the operations $\omega^k: \G_0(G_+)\map\G_0(G_+)$ induced by those in Theorem \ref{Gop} define the same $\lambda$-structure on $\G_0(G_+)\cong A(G)$ as that of Siebeneicher.
\end{prop}

\begin{proof}
The assertion follows from Remark \ref{ksubset} and the argument in Section 8 of \cite{GraLambda}. Note that for any $S_+\in\FG(G_+)$, the ordering on an element of $\Diamond^k S_+$ does not matter up to a $G_+$-module isomorphism.
\end{proof}

\begin{rem}
In general, the $\lambda$-operations on $\G_0(G_+)\cong A(G)$ of Siebeneicher gives it only a pre-$\lambda$-ring structure. However, it becomes a $\lambda$-ring if $G$ is a cyclic group of odd order.
\end{rem}

If $R=F$ is a field and $G$ is a finite group, then the category $\cP(F,G)$ is the same as $\FG(F[G])$. Using split monomorphisms (resp. isomorphisms) as the cofibrations (resp. weak equivalences) in the Waldhausen categories $\cP(F,G)$ and $\FG(F[G])$ we conclude that they have the same Waldhausen $\K$-theory spectra. Now if, in addition, the order of $G$ is invertible in $F$, then $F[G]$ is semisimple and the category $\FG(F[G])$ is the same as the category $\mathcal{P}(F[G])$, which is the category of finitely generated and projective modules over $F[G]$. Consequently, the $\K$-theory spectrum of $\FG(F[G])$ is the algebraic $\K$-theory of the group algebra $F[G]$. In this manner one can construct `illegitimate $\lambda$-operations' on the higher algebraic $\K$-theory of a possibly noncommutative group algebra $F[G]$ (whenever $G$ is nonabelian).

\begin{rem}
 If $G$ is a finite abelian group then Grayson's machinery can be used to directly construct $\lambda$-operations on $\K_i(F[G])$.  However, these $\lambda$-operations will differ from the ones that we just described above since $\boxtimes$ is different in the two cases: in the `illegitimate case' it is $\otimes_F$ with diagonal $G$-action, whereas in the other it is $\otimes_{F[G]}$.
\end{rem}


\bibliographystyle{abbrv}

\bibliography{/home/ibatu/Professional/math/MasterBib/bibliography}

\begin{thebibliography}{10}

\bibitem{Anevski}
S.~Anevski.
\newblock Reconstructing the spectrum of ${F_1}$ from the stable homotopy
  category.
\newblock {\em arXiv:1103.1235}.

\bibitem{AtiTal}
M.~F. Atiyah and D.~O. Tall.
\newblock Group representations, {$\lambda $}-rings and the {$J$}-homomorphism.
\newblock {\em Topology}, 8:253--297, 1969.

\bibitem{AusReg}
M.~Auslander.
\newblock On regular group rings.
\newblock {\em Proc. Amer. Math. Soc.}, 8:658--664, 1957.

\bibitem{BerRog}
H.~S. Bergsaker and J.~Rognes.
\newblock Homology operations in the topological cyclic homology of a point.
\newblock {\em Geom. Topol.}, 14(2):735--772, 2010.

\bibitem{Betley}
S.~Betley.
\newblock On the homotopy groups of {$A(X)$}.
\newblock {\em Proc. Amer. Math. Soc.}, 98(3):495--498, 1986.

\bibitem{BetF1}
S.~Betley.
\newblock Some remarks on absolute mathematics.
\newblock {\em Comm. Algebra}, 38(2):576--587, 2010.

\bibitem{Borger}
J.~Borger.
\newblock Lambda-rings and the field with one element.
\newblock {\em arXiv:0906.3146}.

\bibitem{BunTam}
U.~Bunke and G.~Tamme.
\newblock Regulators and cycle maps in higher-dimensional differential
  algebraic {K}-theory.
\newblock {\em Adv. Math.}, 285:1853--1969, 2015.

\bibitem{CarDouDun}
G.~Carlsson, C.~L. Douglas, and B.~I. Dundas.
\newblock Higher topological cyclic homology and the {S}egal conjecture for
  tori.
\newblock {\em Adv. Math.}, 226(2):1823--1874, 2011.

\bibitem{CLS}
C.~Chu, O.~Lorscheid, and R.~Santhanam.
\newblock Sheaves and {$K$}-theory of {$\mathbb{F}_1$-schemes}.
\newblock {\em Adv. Math.}, 229(4):2239--2286, 2012.

\bibitem{ChuMor}
C.~Chu and J.~Morava.
\newblock On the {A}lgebraic {$K$}-theory of {M}onoids.
\newblock {\em arXiv:1009.3235}.

\bibitem{CC2}
A.~Connes and C.~Consani.
\newblock Schemes over {$\mathbb F_1$} and zeta functions.
\newblock {\em Compos. Math.}, 146(6):1383--1415, 2010.

\bibitem{CC1}
A.~Connes and C.~Consani.
\newblock On the notion of geometry over {$\mathbb F_1$}.
\newblock {\em J. Algebraic Geom.}, 20(3):525--557, 2011.

\bibitem{FunF1}
A.~Connes, C.~Consani, and M.~Marcolli.
\newblock Fun with {$\mathbb F\sb 1$}.
\newblock {\em J. Number Theory}, 129(6):1532--1561, 2009.

\bibitem{CorHaeWalWei}
G.~Corti{\~n}as, C.~Haesemayer, M.~E. Walker, and C.~Weibel.
\newblock Toric varieties, monoid schemes and cdh descent.
\newblock {\em J. Reine Angew. Math. (Crelle)}, 698:1--54, 2015.

\bibitem{DavLue}
J.~F. Davis and W.~L{\"u}ck.
\newblock Spaces over a category and assembly maps in isomorphism conjectures
  in {$K$}- and {$L$}-theory.
\newblock {\em $K$-Theory}, 15(3):201--252, 1998.

\bibitem{Deitmar}
A.~Deitmar.
\newblock Remarks on zeta functions and {$K$}-theory over {${\bf F}_1$}.
\newblock {\em Proc. Japan Acad. Ser. A Math. Sci.}, 82(8):141--146, 2006.

\bibitem{DevHopSmi}
E.~S. Devinatz, M.~J. Hopkins, and J.~H. Smith.
\newblock Nilpotence and stable homotopy theory. {I}.
\newblock {\em Ann. of Math. (2)}, 128(2):207--241, 1988.

\bibitem{DreMackey}
A.~W.~M. Dress.
\newblock Contributions to the theory of induced representations.
\newblock In {\em Algebraic {$K$}-theory, {II}: ``{C}lassical'' algebraic
  {$K$}-theory and connections with arithmetic ({P}roc. {C}onf., {B}attelle
  {M}emorial {I}nst., {S}eattle, {W}ash., 1972)}, pages 183--240. Lecture Notes
  in Math., Vol. 342. Springer, Berlin, 1973.

\bibitem{Dress}
A.~W.~M. Dress.
\newblock Induction and structure theorems for orthogonal representations of
  finite groups.
\newblock {\em Ann. of Math. (2)}, 102(2):291--325, 1975.

\bibitem{DreKuk}
A.~W.~M. Dress and A.~O. Kuku.
\newblock A convenient setting for equivariant higher algebraic {$K$}-theory.
\newblock In {\em Algebraic {$K$}-theory, {P}art {I} ({O}berwolfach, 1980)},
  volume 966 of {\em Lecture Notes in Math.}, pages 59--68. Springer, Berlin,
  1982.

\bibitem{Durov}
N.~Durov.
\newblock New {A}pproach to {A}rakelov {G}eometry.
\newblock {\em arXiv:0704.2030}.

\bibitem{Dwyer}
W.~G. Dwyer.
\newblock Twisted homological stability for general linear groups.
\newblock {\em Ann. of Math. (2)}, 111(2):239--251, 1980.

\bibitem{FarSni}
D.~R. Farkas and R.~L. Snider.
\newblock {$K_{0}$} and {N}oetherian group rings.
\newblock {\em J. Algebra}, 42(1):192--198, 1976.

\bibitem{FarJon}
F.~T. Farrell and L.~E. Jones.
\newblock Isomorphism conjectures in algebraic {$K$}-theory.
\newblock {\em J. Amer. Math. Soc.}, 6(2):249--297, 1993.

\bibitem{FloWei}
J.~Flores and C.~Weibel.
\newblock Picard groups and class groups of monoid schemes.
\newblock {\em J. Algebra}, 415:247--263, 2014.

\bibitem{GeiHes}
T.~Geisser and L.~Hesselholt.
\newblock Topological cyclic homology of schemes.
\newblock In {\em Algebraic {$K$}-theory ({S}eattle, {WA}, 1997)}, volume~67 of
  {\em Proc. Sympos. Pure Math.}, pages 41--87. Amer. Math. Soc., Providence,
  RI, 1999.

\bibitem{Gil1}
H.~Gillet.
\newblock Riemann-{R}och theorems for higher algebraic {$K$}-theory.
\newblock {\em Adv. in Math.}, 40(3):203--289, 1981.

\bibitem{GilletSurv}
H.~Gillet.
\newblock {$K$}-{T}heory and {I}ntersection {T}heory.
\newblock In {\em Handbook of {$K$}-theory. {V}ol. 1, 2}, pages 235--293.
  Springer, Berlin, 2005.

\bibitem{GilGra}
H.~Gillet and D.~R. Grayson.
\newblock The loop space of the {$Q$}-construction.
\newblock {\em Illinois J. Math.}, 31(4):574--597, 1987.

\bibitem{GraLambda}
D.~R. Grayson.
\newblock Exterior power operations on higher {$K$}-theory.
\newblock {\em $K$-Theory}, 3(3):247--260, 1989.

\bibitem{GunSch}
T.~Gunnarsson and R.~Schw{\"a}nzl.
\newblock Operations in {$A$}-theory.
\newblock {\em J. Pure Appl. Algebra}, 174(3):263--301, 2002.

\bibitem{GSVW}
T.~Gunnarsson, R.~Schw{\"a}nzl, R.~M. Vogt, and F.~Waldhausen.
\newblock An un-delooped version of algebraic {$K$}-theory.
\newblock {\em J. Pure Appl. Algebra}, 79(3):255--270, 1992.

\bibitem{HallNoethRing}
P.~Hall.
\newblock Finiteness conditions for soluble groups.
\newblock {\em Proc. London Math. Soc. (3)}, 4:419--436, 1954.

\bibitem{Hiller}
H.~L. Hiller.
\newblock {$\lambda $}-rings and algebraic {$K$}-theory.
\newblock {\em J. Pure Appl. Algebra}, 20(3):241--266, 1981.

\bibitem{HopSmi}
M.~J. Hopkins and J.~H. Smith.
\newblock Nilpotence and stable homotopy theory. {II}.
\newblock {\em Ann. of Math. (2)}, 148(1):1--49, 1998.

\bibitem{HSS}
M.~Hovey, B.~Shipley, and J.~Smith.
\newblock Symmetric spectra.
\newblock {\em J. Amer. Math. Soc.}, 13(1):149--208, 2000.

\bibitem{Iriye}
K.~Iriye.
\newblock The nilpotency of elements of the equivariant stable homotopy groups
  of spheres.
\newblock {\em J. Math. Kyoto Univ.}, 22(2):257--259, 1982/83.

\bibitem{Kratzer}
C.~Kratzer.
\newblock {$\lambda $}-structure en {$K$}-th\'eorie alg\'ebrique.
\newblock {\em Comment. Math. Helv.}, 55(2):233--254, 1980.

\bibitem{Kuku}
A.~Kuku.
\newblock Equivariant higher algebraic {$K$}-theory for {W}aldhausen
  categories.
\newblock {\em Beitr\"age Algebra Geom.}, 47(2):583--601, 2006.

\bibitem{LodWhitehead}
J.-L. Loday.
\newblock Higher {W}hitehead groups and stable homotopy.
\newblock {\em Bull. Amer. Math. Soc.}, 82(1):134--136, 1976.

\bibitem{LodAssembly}
J.-L. Loday.
\newblock {$K$}-th\'eorie alg\'ebrique et repr\'esentations de groupes.
\newblock {\em Ann. Sci. \'Ecole Norm. Sup. (4)}, 9(3):309--377, 1976.

\bibitem{MapF1}
O.~Lorscheid and J.~L. Pe{\~n}a.
\newblock Mapping ${F_1}$-land: an overview of geometries over the field with
  one element.
\newblock {\em Noncommutative geometry, arithmetic, and related topics},
  Proceedings of the 21st meeting of JAMI:241--265, 2011.

\bibitem{Malkiewich}
C.~Malkiewich.
\newblock Coassembly and the {K}-theory of finite groups.
\newblock {\em arXiv:1503.06504, to appear in Adv. Math.}

\bibitem{ManLECT}
Y.~Manin.
\newblock Lectures on zeta functions and motives (according to {D}eninger and
  {K}urokawa).
\newblock {\em Ast\'erisque}, (228):4, 121--163, 1995.
\newblock Columbia University Number Theory Seminar (New York, 1992).

\bibitem{ManinF1}
Y.~I. Manin.
\newblock Cyclotomy and analytic geometry over {${\bf F}_1$}.
\newblock {\em Quanta of Maths}, Clay Mathematics Proceedings 11:385--408,
  2010.

\bibitem{MarCycEnd}
M.~Marcolli.
\newblock Cyclotomy and endomotives.
\newblock {\em P-Adic Numbers, Ultrametric Analysis and Applications},
  1(3):217--263, 2009.

\bibitem{MayTho}
J.~P. May and R.~Thomason.
\newblock The uniqueness of infinite loop space machines.
\newblock {\em Topology}, 17(3):205--224, 1978.

\bibitem{Nishida}
G.~Nishida.
\newblock The nilpotency of elements of the stable homotopy groups of spheres.
\newblock {\em J. Math. Soc. Japan}, 25:707--732, 1973.

\bibitem{ParkerGlDim}
M.~J. Parker.
\newblock On the global dimension of abelian group-rings.
\newblock {\em Quart. J. Math. Oxford Ser. (2)}, 22:495--504, 1971.

\bibitem{QuiAlgK1}
D.~Quillen.
\newblock Higher algebraic {$K$}-theory. {I}.
\newblock In {\em Algebraic {$K$}-theory, {I}: {H}igher {$K$}-theories ({P}roc.
  {C}onf., {B}attelle {M}emorial {I}nst., {S}eattle, {W}ash., 1972)}, pages
  85--147. Lecture Notes in Math., Vol. 341. Springer, Berlin, 1973.

\bibitem{Ravenel}
D.~C. Ravenel.
\newblock {\em Nilpotence and periodicity in stable homotopy theory}, volume
  128 of {\em Annals of Mathematics Studies}.
\newblock Princeton University Press, Princeton, NJ, 1992.
\newblock Appendix C by Jeff Smith.

\bibitem{Salch}
A.~Salch.
\newblock Grothendieck duality under {S}pec {Z}.
\newblock {\em arXiv:1012.0110}.

\bibitem{Scholbach}
J.~Scholbach.
\newblock Algebraic {K}-theory of the infinite place.
\newblock {\em J. Homotopy Relat. Struct.}, 10(4):821--842, 2015.

\bibitem{SchwedeBook}
S.~Schwede.
\newblock Symmetric spectra.
\newblock {\em preprint, available from the author{'}s homepage}.

\bibitem{SchwedeHtpyGp}
S.~Schwede.
\newblock On the homotopy groups of symmetric spectra.
\newblock {\em Geom. Topol.}, 12(3):1313--1344, 2008.

\bibitem{SegGamma}
G.~Segal.
\newblock Categories and cohomology theories.
\newblock {\em Topology}, 13:293--312, 1974.

\bibitem{SegalOp}
G.~Segal.
\newblock Operations in stable homotopy theory.
\newblock In {\em New developments in topology ({P}roc. {S}ympos. {A}lgebraic
  {T}opology, {O}xford, 1972)}, pages 105--110. London Math Soc. Lecture Note
  Ser., No. 11. Cambridge Univ. Press, London, 1974.

\bibitem{Siebeneicher}
C.~Siebeneicher.
\newblock {$\lambda $}-{R}ingstrukturen auf dem {B}urnsidering der
  {P}ermutationsdarstellungen einer endlichen {G}ruppe.
\newblock {\em Math. Z.}, 146(3):223--238, 1976.

\bibitem{SouleF1}
C.~Soul{\'e}.
\newblock Les vari\'et\'es sur le corps \`a un \'el\'ement.
\newblock {\em Mosc. Math. J.}, 4(1):217--244, 312, 2004.

\bibitem{SwanSplitting}
R.~G. Swan.
\newblock A splitting principle in algebraic {$K$}-theory.
\newblock In {\em Representation theory of finite groups and related topics
  ({P}roc. {S}ympos. {P}ure {M}ath., {V}ol. {XXI}, {U}niv. {W}isconsin,
  {M}adison, {W}is., 1970)}, pages 155--159. Amer. Math. Soc., Providence,
  R.I., 1971.

\bibitem{Thomason}
R.~W. Thomason.
\newblock First quadrant spectral sequences in algebraic {$K$}-theory via
  homotopy colimits.
\newblock {\em Comm. Algebra}, 10(15):1589--1668, 1982.

\bibitem{ThoTro}
R.~W. Thomason and T.~Trobaugh.
\newblock Higher algebraic {$K$}-theory of schemes and of derived categories.
\newblock In {\em The Grothendieck Festschrift, Vol.\ III}, volume~88 of {\em
  Progr. Math.}, pages 247--435. Birkh\"auser Boston, Boston, MA, 1990.

\bibitem{Tits}
J.~Tits.
\newblock Sur les analogues alg\'ebriques des groupes semi-simples complexes.
\newblock In {\em Colloque d'alg\`ebre sup\'erieure, tenu \`a {B}ruxelles du 19
  au 22 d\'ecembre 1956}, Centre Belge de Recherches Math\'ematiques, pages
  261--289. \'Etablissements Ceuterick, Louvain, 1957.

\bibitem{TVF1}
B.~To{\"e}n and M.~Vaqui{\'e}.
\newblock Au-dessous de {${\rm Spec}\,\mathbb Z$}.
\newblock {\em J. K-Theory}, 3(3):437--500, 2009.

\bibitem{tomDieck}
T.~tom Dieck.
\newblock Orbittypen und {\"a}quivariante homologie {II}.
\newblock {\em Arch. Math. (Basel)}, 26(6):650--662, 1975.

\bibitem{Wal1}
F.~Waldhausen.
\newblock Algebraic {$K$}-theory of generalized free products. {I}, {II}.
\newblock {\em Ann. of Math. (2)}, 108(1):135--204, 1978.

\bibitem{WalOp}
F.~Waldhausen.
\newblock Operations in the algebraic {$K$}-theory of spaces.
\newblock In {\em Algebraic {$K$}-theory, {P}art {II} ({O}berwolfach, 1980)},
  volume 967 of {\em Lecture Notes in Math.}, pages 390--409. Springer, Berlin,
  1982.

\bibitem{Waldhausen}
F.~Waldhausen.
\newblock Algebraic {$K$}-theory of spaces.
\newblock In {\em Algebraic and geometric topology ({N}ew {B}runswick,
  {N}.{J}., 1983)}, volume 1126 of {\em Lecture Notes in Math.}, pages
  318--419. Springer, Berlin, 1985.

\bibitem{WeibelKBook}
C.~Weibel.
\newblock The {K}-book: an introduction to algebraic {K}-theory.
\newblock {\em Graduate Studies in Math.}, 145, AMS, 2013.

\bibitem{WeiYao}
C.~Weibel and D.~Yao.
\newblock Localization for the {$K$}-theory of noncommutative rings.
\newblock In {\em Algebraic {$K$}-theory, commutative algebra, and algebraic
  geometry ({S}anta {M}argherita {L}igure, 1989)}, volume 126 of {\em Contemp.
  Math.}, pages 219--230. Amer. Math. Soc., Providence, RI, 1992.

\bibitem{WeiWil}
M.~Weiss and B.~Williams.
\newblock Assembly.
\newblock {\em Proc. Conf. Novikov Conjectures, Index Theorems and Rigidity II,
  Oberwolfach 1993, Lecture Notes Ser. 227, Cambridge University Press}, pages
  332--352, 1995.

\end{thebibliography}

\vspace{5mm}

\end{document}